\documentclass[a4paper,11pt]{scrartcl}
\usepackage[utf8]{inputenc}
\usepackage{latexsym,amsmath,amsfonts,amscd,amssymb,amsthm,mathrsfs}
\usepackage{pifont}
\usepackage{eucal}
\usepackage[all]{xy}
\usepackage{color}
\definecolor{Myblue}{rgb}{0,0,0.6}
\usepackage[colorlinks,citecolor=Myblue,linkcolor=Myblue,urlcolor=Myblue,pdfpagemode=UseNone,backref = page]{hyperref}

\theoremstyle{plain}
\newtheorem{theorem}{Theorem}[section]
\newtheorem{proposition}[theorem]{Proposition}
\newtheorem{lemma}[theorem]{Lemma}
\newtheorem{corollary}[theorem]{Corollary}

\theoremstyle{definition}
\newtheorem{definition}[theorem]{Def{}inition}
\newtheorem{example}[theorem]{Example}
\theoremstyle{remark}
\newtheorem{remark}[theorem]{Remark}

\renewcommand{\a}{\alpha}

\newcommand{\Id}{\mathrm{Id}}   
\newcommand{\SO}{\mathrm{SO}}
\newcommand{\bC}{\mathbb{C}}

\newcommand{\bR}{\mathbb{R}}
\newcommand{\bQ}{\mathbb{Q}}
\newcommand{\bZ}{\mathbb{Z}}

\newcommand{\cD}{\mathcal{D}}
\newcommand{\cF}{\mathcal{F}}
\newcommand{\cH}{\mathcal{H}}
\newcommand{\cL}{\mathcal{L}}
\newcommand{\cQ}{\mathcal{Q}}
\newcommand{\fg}{\mathfrak{g}}

\newcommand{\fs}{\mathfrak{s}}
\newcommand{\ft}{\mathfrak{t}}
\newcommand{\fX}{\mathfrak{X}}
\newcommand{\N}{\nabla}
\newcommand{\ox}{\otimes}

\newenvironment{keywords}%
   {\begin{trivlist}\item[]{\bfseries\sffamily Keywords:}\ }
   {\end{trivlist}}

\newenvironment{MSC}%
   {\begin{trivlist}\item[]{\bfseries\sffamily MSC Classification [2010]:}\ }
   {\end{trivlist}}

\begin{document}

\title{K-cosymplectic manifolds}
 \author{Giovanni Bazzoni$^*$ \quad Oliver Goertsches$^\dagger$
 \\[0.5cm]
  \small{\tt \href{mailto:gbazzoni@math.uni-bielefeld.de}{gbazzoni@math.uni-bielefeld.de}} \qquad
  \small{\tt \href{mailto: oliver.goertsches@math.uni-hamburg.de }{ oliver.goertsches@math.uni-hamburg.de }}\\[0.1cm]
  {\normalsize\slshape $^*$Fakult\"{a}t f\"{ur} Mathematik, Universit\"{a}t Bielefeld, }\\[-0.1cm]
  {\normalsize\slshape Postfach 100301, D-33501 Bielefeld}\\[-0.1cm]
  {\normalsize\slshape $^\dagger$Fachbereich Mathematik, Universit\"{a}t Hamburg, }\\[-0.1cm]
  {\normalsize\slshape Bundesstra\ss e 55, D-20146 Hamburg}\\[-0.1cm]}
\date{}


 \maketitle

\begin{abstract} In this paper we study K-cosymplectic manifolds, i.e., smooth cosymplectic manifolds for which the Reeb f{}ield is Killing with respect to some Riemannian metric. These structures
generalize coK\"ahler structures, in the same way as K-contact structures generalize Sasakian structures. In analogy to the contact case, we distinguish between (quasi-)regular and irregular
structures; in the regular case, the K-cosymplectic manifold turns out to be a 
flat circle bundle over an almost K\"ahler manifold. We investigate de Rham and basic cohomology of K-cosymplectic manifolds, as well as cosymplectic and Hamiltonian vector f{}ields and group actions
on such manifolds. The deformations of type I and II in the contact setting have natural analogues for cosymplectic manifolds; those of type I can be used to show that compact K-cosymplectic manifolds
always carry quasi-regular structures. We consider Hamiltonian group actions and use the momentum map to study the equivariant cohomology of the canonical torus action on a compact K-cosymplectic
manifold, resulting in relations between the basic cohomology of the characteristic foliation and the number of closed Reeb orbits on an irregular K-cosymplectic manifold.
\end{abstract}


\begin{MSC}
 Primary 53C25, 53D15; secondary 53C12, 53D20, 55N91.
\end{MSC}

\section{Introduction}

In \cite{GH}, Gray and Hervella organized almost Hermitian structures on even dimensional manifolds
in 16 classes; this was done in order to generalize K\"ahler geometry by requiring only the presence of an almost complex
structure and of a compatible Riemannian metric. The same approach was used by Chinea and Gonz\'alez in \cite{CG} to classify almost
contact metric structures, the odd-dimensional analogue of almost Hermitian structures. The authors show 
how almost contact metric structures can be set apart in a certain number\footnote{4096, to be precise.} of classes. As in the Gray-Hervella paper, the authors of \cite{CG} used the covariant
derivative of the fundamental form of an almost contact metric stucture. In \cite{CM}, instead, almost contact metric structures are classif{}ied using the Nijenhuis tensor.

For us, an almost contact metric structure on an odd-dimensional manifold $M$, consisting of a $1$-form $\eta$, a tensor $\phi\colon TM\to TM$, an adapted Riemannian metric $g$, and the associated
K\"ahler form $\omega$, is called 
\emph{cosymplectic} if $\eta$ and $\omega$ are closed. A cosymplectic structure determines a Poisson structure in a canonical way: every leaf of the foliation $\cF=\ker\eta$ is endowed with a
symplectic structure, given by the pullback of $\omega$.

Cosymplectic structures were introduced by Goldberg and Yano in \cite{GY} under the name of \emph{almost cosymplectic structures}. In fact,
until very recently, the word cosymplectic indicated what we call here a coK\"ahler structure (see \cite{Bl2,Blair,CdLM,FV}). The terminology used in this paper was introduced in \cite{Li} and 
seems to have been adopted since (see \cite{BLO,bazzoni_oprea,CMdNY,GMP}).

While K\"ahler structures play a prominent role in even dimension, it is not so clear which one, among almost contact metric structures, should be taken to be the odd dimensional analogue of K\"ahler
structures. Of all possible candidates, Sasakian and coK\"ahler structures seem to be the most natural. Sasakian geometry has been a very active research area for quite a long time; by now, the
standard reference is \cite{Boyer_Galicki}.

CoK\"ahler (and, more generally, cosymplectic) geometry has also drawn a great deal of interest in the last years (see for instance \cite{BFM,BLO,bazzoni_oprea,FV,GMP,Li}). Very recently, cosymplectic
structures have appeared in a natural way in the study of $b$-symplectic (or $log$-symplectic) structures (see \cite{C,FMM,GMP2,GMS}). We refer to \cite{CMdNY} for a
nice overview on cosymplectic geometry and its connection with other areas of mathematics (specially geometric mechanics) as well as with physics.

In this paper we consider cosymplectic structures for which the Reeb vector f{}ield is Killing. In analogy with the terminology used in the contact metric setting, we call such structures
\emph{K-cosymplectic}. Every coK\"ahler structure is K-cosymplectic, but we will give dif{}ferent examples of K-cosymplectic non coK\"ahler structures.

Apart from cosymplectic structures, there is another possible meaning of cosymplectic, introduced by Libermann in \cite{Libermann}. A \emph{manifold}
$M^{2n+1}$ is \emph{cosymplectic} if it is endowed with a 1-form $\eta$ and a 2-form $\omega$, both closed, such that $\eta\wedge\omega^n\neq 0$ at every point of $M$. 
The Reeb f{}ield is uniquely determined by the conditions $\imath_\xi\omega=0$ and $\imath_\xi\eta=1$. Starting with this, we can def{}ine a cosymplectic manifold $(M,\eta,\omega)$ to be 
\textit{K-cosymplectic} if there exists a Riemannian metric $g$ on $M$ for which the Reeb f{}ield is Killing. Roughly speaking, the dif{}ference  between the two def{}initions of K-cosymplectic is
that in the f{}irst case we f{}ix a metric on $M$, while in the second one we do not. This is akin to the dif{}ference between almost K\"ahler and symplectic geometry.
Despite this apparent discrepancy, Proposition \ref{equivalence:1} shows that these def{}initions are equivalent.

In Sections \ref{section:Structure} and \ref{section:deformations}, for example, we deal with K-cosymplectic structures, whereas in Sections \ref{section:forms}, \ref{section:HamiltonianVF},
\ref{section:HamiltonianGroupActions} and \ref{section:ClosedReebOrbits}, we do not choose a metric; we simply use the fact that our cosymplectic manifold admits some metric for which the Reeb
f{}ield is Killing.

The overall theme of this paper is to develop a theory of K-cosymplectic manifolds, motivated by the existing theory of K-contact manifolds. We introduce several useful concepts in analogy, such as a distinction between regular, 
quasi-regular and irregular structures, see Section \ref{section:Structure}. There, we also prove a structure theorem for compact K-cosymplectic manifolds.
In Section \ref{section:forms} we show that the de Rham and the basic (with respect to the characteristic foliation) cohomology of a compact K-cosymplectic 
manifold contain the same information. Section \ref{section:HamiltonianVF} is devoted to cosymplectic and Hamiltonian vector f{}ields. These have already been introduced in \cite{Albert}, and used to
prove a reduction result \emph{\`a la Marsden-Weinstein} for cosymplectic manifolds. In Section \ref{section:deformations} we introduce a generalization of the deformations of type I 
(we briefly comment on type II deformations in Example \ref{ex:typeIIdeformation}) for general almost contact metric structures, and show that these deformations respect (K-)cosymplectic and coK\"ahler structures. On the one hand 
they can be used to show that any compact manifold with a K-cosymplectic structure admits a quasi-
regular K-cosymplectic structure, see Proposition \ref{prop:exquasireg}; on the other hand there are situations in which they can be used to deform K-cosymplectic structures in such a way that one has
only f{}initely many closed Reeb orbits. This is the topic of Section \ref{section:ClosedReebOrbits}, in which we use equivariant cohomology to investigate the relation between closed Reeb orbits and
basic cohomology on a compact K-cosymplectic manifold, similar as it exists in the K-contact case \cite{GNT}: for a compact K-cosymplectic manifold with f{}irst Betti number $b_1=1$, the number of
closed Reeb orbits is either inf{}inite or given by the total basic Betti number of $M$, see Corollary \ref{cor:dimC}. The main ingredient in the proof are Morse-Bott properties of the cosymplectic
momentum map, as they are shown in Section \ref{section:HamiltonianGroupActions}. We conclude with  examples of irregular coK\"ahler structures on the product of an odd complex quadric with $S^1$ with
minimal number of closed Reeb orbits.

\section{K-cosymplectic manifolds and structures}\label{section:Generalities}

An \textbf{almost contact metric structure} (see \cite{Blair}) on an odd-dimensional smooth manifold $M$ consists of:
\begin{enumerate}
\item[i)] a 1-form $\eta$ and a vector f{}ield $\xi$, the Reeb f{}ield, such that $\eta(\xi)\equiv 1$;
\item[ii)] a tensor $\phi\colon TM\to TM$ such that $\phi^2=-\Id+\eta\ox\xi$;
\item[iii)] a Riemannian metric $g$ such that $g(\phi X,\phi Y)=g(X,Y)-\eta(X)\eta(Y)$ for $X,Y\in\fX(M)$.
\end{enumerate}

Conditions i) and ii) imply $\eta\circ\phi=\phi\xi=0$. Plugging $Y=\xi$ into iii) gives $\eta(X)=g(X,\xi)$ $\forall X\in\fX(M)$. Given an almost contact metric structure $(\eta,\xi,\phi,g)$ on a
$2n+1$-dimensional manifold $M$, its \textbf{K\"ahler form} $\omega\in\Omega^2(M)$ is given by $\omega(X,Y)=g(X,\phi Y)$. This implies that $\eta\wedge\omega^n$ is a volume form on $M$. The almost
contact metric structure $(\eta,\xi,\phi,g)$ is 
\begin{itemize}
\item \textbf{cosymplectic} if $d\eta=d\omega=0$;
\item \textbf{contact metric} if $\omega=d\eta$;
\item \textbf{normal} if $N_\phi+d\eta\ox\xi=0$;
\item \textbf{coK\"ahler} if it is cosymplectic and normal;
\item \textbf{Sasakian} if it is contact metric and normal.
\end{itemize}
Given a tensor $A\colon TM\to TM$, 
its \textbf{Nijenhuis torsion} $N_A\colon \wedge^2 TM\to TM$ is def{}ined as
\[
N_A(X,Y)=A^2[X,Y]-A([AX,Y]+[X,AY])+[AX,AY].
\]

\begin{definition}
A cosymplectic structure $(\eta,\xi,\phi,g)$ on a manifold $M$ is \textbf{K-cosymplectic} if the Reeb f{}ield is Killing.
\end{definition}

Let us consider the following four tensors on an almost contact metric manifold:
\begin{itemize}
\item $N^{(1)}(X,Y)=N_\phi(X,Y)+d\eta(X,Y)\xi$;
\item $N^{(2)}(X,Y)=(L_{\phi X}\eta)(Y)-(L_{\phi Y}\eta)(X)$;
\item $N^{(3)}(X)=(L_{\xi}\phi)(X)$;
\item $N^{(4)}(X)=(L_{\xi}\eta)(X)$.
\end{itemize}

It is well known (see \cite[Theorem 6.1]{Blair}) that if $N^{(1)}=0$, i.e. if the almost contact metric structure is normal, then $N^{(i)}=0$, for $i=2,3,4$. 
We collect in the following proposition some properties of $N^{(i)}$ for cosymplectic almost contact metric structures. We refer to \cite[Section 3]{CMdNY} for the proofs.

\begin{proposition}
Assume that the structure $(\eta,\xi,\phi,g)$ is cosymplectic. Then $N^{(2)}=N^{(4)}=0$ and 
\[
\N\xi=-\frac{1}{2}\phi\circ N^{(3)},
\]
where $\N$ is the Levi-Civita connection of $g$. In particular, $N^{(3)}$ vanishes if and only if $\xi$ is a parallel vector f{}ield.
\end{proposition}

A parallel vector f{}ield on a Riemannian manifold is always Killing, hence the vanishing of $N^{(3)}$ is a suf{}f{}icient condition for the cosymplectic structure to be K-cosymplectic.
But it is also necessary, as the following Proposition shows:

\begin{proposition}
Let $(\eta,\xi,\phi,g)$ be a K-cosymplectic structure. Then $N^{(3)}=0$.
\end{proposition}
We can therefore give an equivalent characterization of K-cosymplectic structures:
\begin{corollary}\label{characterization}
A cosymplectic structure $(\eta,\xi,\phi,g)$ on a manifold $M$ is K-cosymplectic if and only if $N^{(3)}=0$. In particular, the Reeb f{}ield is parallel on a K-cosymplectic manifold.
\end{corollary}

\begin{remark}\label{eta_parallel}
If $(\eta,\xi,\phi,g)$ is K-cosymplectic, $\eta(X)=g(X,\xi)$ and $\N\xi=0$ imply $\N\eta=0$. The converse is also true: if $(\eta,\xi,\phi,g)$ is cosymplectic and $\N\eta=0$, then $\xi$ is parallel,
hence Killing, and $(\eta,\xi,\phi,g)$ is K-cosymplectic. Thus we see that K-cosymplectic structures are precisely the almost contact metric structures in the class $\mathcal{C}_2$ of \cite{CG}.
\end{remark}

\begin{remark}
The same terminology is used in the context of contact metric structures, where the name K-contact is used when the Reeb f{}ield is Killing, which is
equivalent to $N^{(3)}=0$.
\end{remark}

As mentioned in the introduction there is a second notion of cosymplectic:
\begin{definition} A $2n+1$-dimensional manifold $M$ is \textbf{cosymplectic} if it is endowed with a $1$-form $\eta$ and a $2$-form $\omega$, both closed, such that $\eta\wedge \omega^n$ is a volume
form. The \textbf{Reeb f{}ield} is uniquely determined by $\eta(\xi)=1$ and $\imath_\xi\omega=0$. We call $M$ \textbf{K-cosymplectic} if there exists a Riemannian metric on $M$ for which $\xi$ is a
Killing vector f{}ield.
\end{definition}

Given a cosymplectic manifold $(M,\eta,\omega)$, the $2$-form $\omega$ induces a symplectic structure on $\cD=\ker\eta$. It is well known that in this case, there exist a Riemannian 
metric $g_\cD$ and an almost complex structure $J$ on $\cD$ such that $g_\cD(X,Y)=\omega(JX,Y)$ and $g_\cD(JX,JY)=g_\cD(X,Y)$ for all $X,Y$ tangent to $\cD$ (see \cite{Audin}). The pair $(g_\cD,J)$ is said to be compatible with 
$\omega$. Extend $g_\cD$ to the whole tangent bundle by requiring $\xi$ and $\cD$ to be orthogonal and $\xi$ to have length $1$. We get a Riemannian metric $g=g_\cD+\eta\ox\eta$, which is called
\textbf{adapted} to the 
cosymplectic structure $(\eta,\omega)$. 
Such an adapted metric is not unique, as it depends on the choice of $g_\cD$.

We thus have two meanings of the word K-cosymplectic. We show next that these two def{}initions are equivalent.

\begin{proposition}\label{equivalence:1}
 Let $(M,\eta,\omega)$ be a cosymplectic manifold and let $\xi$ be the Reeb f{}ield. The following are equivalent:
 \begin{enumerate}
  \item $(M,\eta,\omega)$ is K-cosymplectic;
  \item $M$ admits an adapted Riemannian metric $g$ for which $\xi$ is Killing;
  \item $M$ carries a K-cosymplectic structure $(\eta,\xi,\phi,g)$.
 \end{enumerate}
\begin{proof}
 Suppose $(M,\eta,\omega)$ is a K-cosymplectic manifold and let $\bar{g}$ be a metric for which $\xi$ is Killing. Following \cite{Ruk_1}, we show how to construct 
 an adapted metric $g$ for which $\xi$ is Killing. As a byproduct, we will also get a tensor $\phi$ such that $(\eta,\xi,\phi,g)$ is K-cosymplectic with K\"ahler form $\omega$, proving basically that a K-cosymplectic manifold 
 $(M,\eta,\omega)$ with an adapted Riemannian metric $g$ for which $\xi$ is Killing carries a K-cosymplectic structure. This will give $1.\Rightarrow 2.\Rightarrow 3.$ altogether ($3.\Rightarrow 1.$ is obvious).
 
 First of all we def{}ine a new Riemannian metric $\tilde{g}$ by 
 \[
 \tilde{g}(X,Y)=\frac{\bar{g}(X,Y)}{\bar{g}(\xi,\xi)} - \left( \frac{\bar{g}(\xi,X)\bar{g}(\xi,Y)}{\bar{g}(\xi,\xi)^2} - \eta(X)\eta(Y)\right).
 \]
 This metric has the property that $\tilde{g}(\xi,\xi)=1$ (note that it is positive def{}inite on $\ker \eta$ because of the Cauchy-Schwarz inequality) and that $\xi$ is still a Killing vector f{}ield
with  respect to $\tilde{g}$. It satisf{}ies $\tilde{g}(\xi,X) = \eta(X)$ for all $X$; in particular, $\cD=\ker \eta$ equals the orthogonal complement of $\xi$ with respect to $\tilde{g}$.
 
 Def{}ine a tensor $A$ on $M$ by $\omega(X,Y)=\tilde{g}(AX,Y)$, for $X,Y\in\fX(M)$. Then $A\xi=0$ and $L_\xi A=0$. By the properties of $\tilde{g}$, the distribution $\cD$ is invariant under $A$. By
polar decomposition (see \cite{Audin,Blair}), we obtain $A=\phi\circ B$, where $B=\sqrt{A^tA}$. Because $A$ is skew-symmetric with respect to $\tilde{g}$, it is in particular normal, and it follows
that $\phi$ and $B$ commute. Moreover, $B$ is given by $B=\sqrt{A^tA}=\sqrt{-A^2}$; hence, $B^2 = -A^2 = -\phi^2B^2$, which implies that $\phi$ restricts to an almost complex structure on $\cD$. Since
$B^2=- A^2$, taking the Lie derivative in the $\xi$
 direction we get
\[
B\circ L_\xi B+L_\xi B\circ B=0.
\]
Let $X\in\fX(M)$ be an eigenvector f{}ield of $B$ with eigenfunction $e^\lambda$. Then
\[
B(L_\xi B(X))=-e^\lambda L_\xi B(X),
\]
and $L_\xi B(X)$ is an eigenvector f{}ield of $B$ with negative eigenfunction $-e^\lambda$. This is impossible, hence $L_\xi B=0$. This is equivalent to $L_\xi\phi=0$, since $L_\xi A=0$. Def{}ine a
metric
$g$ on $M$ by
\begin{equation}\label{adapted}
\omega(X,Y)=g(\phi X,Y) \quad \mathrm{and} \quad g(\xi,\xi)=1. 
\end{equation}
This implies that $g$ is adapted to $(\eta,\omega)$. Since $L_\xi\omega=L_\xi\phi=0$, \eqref{adapted} also shows that $\xi$ is Killing with respect to $g$.
It is immediate to check that $(\eta,\xi,\phi,g)$ is an almost contact metric structure with K\"ahler form $\omega$; since $d\eta=d\omega=0$, it is a 
cosymplectic structure. F{}inally, $L_\xi\phi=N^{(3)}=0$, hence the structure is K-cosymplectic by Corollary \ref{characterization}.
\end{proof}
\end{proposition}

In the compact case, we can characterize K-cosymplectic manifolds in terms of the existence of a certain torus action on $M$. 
More precisely, suppose that $(M,\eta,\omega)$ is a compact K-cosymplectic manifold. Let $\xi$ be the Reeb f{}ield and $g$ be an adapted metric for which $\xi$ is Killing (such a metric exists by
Proposition \ref{equivalence:1}). By the Myers-Steenrod theorem (see \cite{MS}), the isometry group of 
$(M,g)$ is a compact Lie group $G$. The closure of the Reeb flow in $G$ is a torus. Hence, $M$ is endowed with a smooth torus action. In \cite{Yamazaki}, Yamazaki proved, in the context of K-contact geometry, that the existence 
of such a torus action on a compact contact manifold $(M,\eta)$ completely characterizes K-contactness. The same statement holds for compact K-cosymplectic manifolds. Namely we have:
\begin{proposition}\label{Yamazaki}
 Let $(M,\eta,\omega)$ be a compact cosymplectic manifold. The following are equivalent:
 \begin{enumerate}
  \item $(M,\eta,\omega)$ is K-cosymplectic;
  \item there exists a torus $T$, a smooth ef{}fective $T$-action $h\colon T\times M\to M$ and a homomorphism $\Psi\colon\bR\to T$ with dense image such that $\psi_t = h_{\Psi(t)}$, where $\psi$ is the
flow of the Reeb f{}ield.
 \end{enumerate}
\end{proposition}
\begin{proof}
 The proof is formally analogous to that of \cite[Proposition 2.1]{Yamazaki}.
\end{proof}

According to \cite{GY}, K-cosymplectic geometry coincides with coK\"ahler geometry in dimension 3. However, in higher dimension, K-cosymplectic structures strictly generalize coK\"ahler structures. 
To see this, recall that compact coK\"ahler manifolds satisfy very strong topological properties (see \cite{bazzoni_oprea,CdLM}). We collect them:
\begin{proposition}\label{compact_coKahler}
Let $M$ be a compact manifold endowed with a coK\"ahler structure. Then 
\begin{itemize}
 \item the f{}irst Betti number of $M$ is odd;
 \item the fundamental group of $M$ contains a f{}inite index subgroup of the form $\Gamma\times\bZ$, where $\Gamma$ is the fundamental group of a K\"ahler manifold;
 \item $M$ is formal in the sense of Sullivan.
\end{itemize}
\end{proposition}

\begin{remark}\label{first_Betti_number}
A compact manifold $M$ endowed with a cosymplectic structure must have $b_1(M)\geq 1$. Indeed, the 1-form $\eta$ is closed; if it was exact, $\eta=df$, then it should
vanish at some point $p\in M$, because of compactness. But this is impossible, since $\eta(\xi)=1$ implies that $\eta$ is nowhere vanishing. Hence $[\eta]\neq 0$ in $H^1(M;\bR)$. In particular, a
compact cosymplectic manifold is never simply connected. Apart from this, the f{}irst Betti number of a compact cosymplectic manifold is not constrained; for instance in \cite[Section 5]{BFM}, the
authors study the geography of compact cosymplectic manifolds and show that for every pair $(k,b)$ with $k\geq 2$ and $b\geq 1$, there exists a compact cosymplectic non-formal manifold of dimension
$2k+1$ with $b_1=b$.
\end{remark}

The next result allows us to construct many K-cosymplectic manifolds. We need to recall two notions. First, given a topological space $X$ and a homeomorphism $\varphi\colon X\to X$, the
\textbf{mapping torus} $X_\varphi$ is by def{}inition the quotient space
\[
X_\varphi =\frac{X\times [0,1]}{(x,0)\sim (\varphi(x),1)};
\]
the projection onto the second factor endows $X_\varphi$ with the structure of a f{}ibre bundle with base $S^1$ and f{}ibre $X$.

An \textbf{almost K\"ahler} manifold is a triple $(K,\tau,h)$ where $(K,\tau)$ is a symplectic manifold and $h$ is a Riemannian metric such that the tensor $J\colon TK\to TK$ def{}ined by the formula
\[
h(X,JY)=\tau(X,Y)
\]
satisf{}ies $J^2=-\Id$.
\begin{proposition}\label{almost_Kahler_mapping_torus}
 Let $(K,\tau,h)$ be an almost K\"ahler manifold and let $\varphi\colon K\to K$ be a dif{}feomorphism such that $\varphi^*\tau=\tau$ and $\varphi^*h=h$. Then the mapping torus $K_\varphi$ carries a
natural K-cosymplectic structure.
\end{proposition}
\begin{proof}
 Let $p\colon K_\varphi\to S^1$ denote the natural projection. Let $d\theta$ be the angular form on $S^1$ and set $\eta=p^*(d\theta)$. Consider the projection $\pi\colon K\times [0,1]\to K$ and use it to pull back the symplectic form 
 $\tau$ to $K\times [0,1]$. Since $\varphi^*\tau=\tau$, this gives a 2-form $\omega\in\Omega^2(K_\varphi)$ which is closed and satisf{}ies $\omega^n\neq 0$, where $2n=\dim K$. It is clear that
$\eta\wedge\omega^n\neq 0$. 
 Def{}ine $\xi$ by $\imath_\xi\omega=0$ and $\eta(\xi)=1$. Then $\xi$ projects to the tangent vector to $S^1$. Consider an interval $I\subset S^1$ such that $p^{-1}(I)\cong I\times K$. Endow $I\times
K$ with the product metric 
 $(dt)^2+h$. Also, def{}ine $\phi\colon T(I\times K)\to T(I\times K)$ by $\phi X=JX$ for $X\in\fX(K)$ and $\phi\xi=0$. Then $\xi$ has length 1 and satisif{}ies (globally) $\phi\xi=0$.
 Since $\varphi^*h=h$ and $\varphi^*\tau=\tau$, we can glue the local metric and the local tensor $\phi$ to global objects $g$ and $\phi$. By construction, $\xi$ is Killing for the metric $g$.
 The structure $(\eta,\xi,\phi,g)$ is K-cosymplectic.
 \end{proof}
 
 \begin{remark}\label{sympl_mapping_torus}
Similar to Proposition \ref{almost_Kahler_mapping_torus} one has (see \cite[Lemmata 1 and 4]{Li}):
\begin{itemize}
\item  if $(K,\tau)$ is a symplectic manifold and $\varphi\colon K\to K$ is a symplectomorphism, then $K_\varphi$ carries a natural cosymplectic structure;
\item  if $(K,\tau,h)$ is a K\"ahler manifold and $\varphi\colon K\to K$ is a Hermitian isometry, then $K_\varphi$ carries a natural coK\"ahler structure.
\end{itemize}
Recall that a Hermitian isometry of a K\"ahler manifold is a biholomorphism of the underlying complex 
manifold which is also a Riemannian isometry, hence a symplectomorphism. In all such cases, if $p\colon K_\varphi\to S^1$ denotes the mapping torus projection, the 1-form $\eta$ on $K_\varphi$ is
simply the pullback of the angular form $d\theta$ on $S^1$, hence $[\eta]\in H^1(K_\varphi;\bZ)$.
\end{remark}
 
\begin{corollary}\label{product}
 Let $(K,\tau,h)$ be an almost K\"ahler manifold. Then $M=K\times S^1$ admits a natural K-cosymplectic structure.
\end{corollary}
\begin{proof}
 $K\times S^1$ is the mapping torus of the identity.
\end{proof}

\begin{example}\label{non_formal} The Kodaira-Thurston manifold $(KT,\omega)$ is a compact symplectic manifold with $b_1(KT)=3$ (see \cite[Example 2.1]{OT}); in particular, $KT$ is not K\"ahler. One
can choose a Riemannian metric $g$ on $KT$ in such a way that $(KT,\omega,g)$ is an almost K\"ahler manifold. Set $M=KT\times S^1$. Then $M$ is a compact K-cosymplectic manifold with $b_1(M)=4$,
hence $M$ is not coK\"ahler.
\end{example}

\begin{example}
By Remark \ref{first_Betti_number}, the f{}irst Betti number of a compact cosymplectic manifold is $\geq 1$. We give an example of a manifold $M$ with
$b_1(M)=1$, endowed with a K-cosymplectic structure which is not coK\"ahler. Let $N$ be the compact, 8-dimensional simply connected symplectic non-formal manifold constructed in \cite{FM}. Let
$\omega$ denote the symplectic form on $N$. Endow $N$ with a Riemannian metric $\bar{g}$ adapted to $\omega$; then 
$(N,\omega,\bar{g})$ is almost K\"ahler and $M=N\times S^1$ is a K-cosymplectic manifold with $b_1(M)=1$. Notice that $M$ is non-formal because $N$ is, hence $M$ is not coK\"ahler.
\end{example}

More generally, using the same argument can prove the following proposition:
\begin{proposition}
Let $(K,\omega)$ be a compact symplectic non-formal manifold. Endow $M=K\times S^1$ with the natural product K-cosymplectic structure. Then $M$ is K-cosymplectic but not coK\"ahler.
\end{proposition}

Further examples of compact K-cosymplectic non coK\"ahler manifolds can be constructed using the fact that, by a result of Gompf (see \cite{Gompf}), every f{}initely presentable group is the
fundamental group of a symplectic 4-manifold.
By applying Corollary \ref{product}, we can obtain K-cosymplectic manifolds with fundamental group $\Gamma\times\bZ$, where $\Gamma$ is any f{}initely presentable group. But we have noticed above that
the fundamental group of a compact coK\"ahler manifold can not be arbitrary.

\section{Regular K-cosymplectic structures}\label{section:Structure}

Let $(\eta,\xi,\phi,g)$ be an almost contact metric structure on $M$. Consider the distribution $\ker \eta$. By the Frobenius Theorem, a distribution integrates to a foliation $\cF$ if and only if
the integrability condition $\eta\wedge d\eta=0$ is satisf{}ied. As we already observed, if $(\eta,\xi,\phi,g)$ is cosymplectic then $d\eta=0$, hence $\eta\wedge d\eta=0$. As a consequence, if
$(\eta,\xi,\phi,g)$ is cosymplectic, $\ker\eta$ integrates to a foliation of codimension 1 on $M$, which we call \textbf{vertical foliation}. In the case of a cosymplectic manifold $(M,\eta,\omega)$, 
we def{}ine the vertical foliation again as $\ker\eta$.

For codimension 1 foliations (not necessarily given as the kernel of a 1-form), we use the following notion of regularity.

\begin{definition}
 Let $M$ be a smooth $n-$dimensional manifold endowed with a codimension-1 foliation $\cF$. Then $\cF$ is \textbf{regular} if every point $p\in M$ has a cubic coordinate neighborhood $U$, 
 with coordinates $(x^1,\ldots,x^n )$, such that $x^i(p)=0$ for every $i$ and such that the leaf $L_p$ of $\cF$ passing through $p$ is given by the 
 equation $x^1=0$.
\end{definition}

There are two simple situations in which a codimension 1 foliation is regular:

\begin{example}\label{Tischler}
 Suppose that $M$ is a compact manifold of dimension $n$ endowed with a closed and nowhere vanishing 1-form $\sigma$. This implies that $[\sigma]\neq 0\in H^1(M;\bR)$ and that $\cF=\ker\sigma$ is a foliation. 
 Assume further that $[\sigma]$ lies in $H^1(M;\bZ)\subset H^1(M;\bR)$.
 By a result of Tischler (see \cite{Ti}), $M$ f{}ibres over the circle; the map $p\colon M\to S^1$ is given precisely by $[\sigma]$ under the usual correspondence
$H^1(M;\bZ)\cong[M,K(\bZ,1)]=[M,S^1]$.
 In particular, $\sigma=p^*(d\theta)$, 
 where $d\theta$ is the angular form on $S^1$. One can also show that $M$ is, in fact, the mapping torus $N_\varphi$ of a dif{}feomorphism $\varphi\colon N\to N$, where $N$ is a smooth, compact
manifold of dimension $n-1$. In this case, the 
 f{}ibres of the projection $p\colon M=N_\varphi\to S^1$ coincide with the leaves of the foliation $\cF$, which is therefore regular.
\end{example} 

Tischler's argument also works when $[\sigma]\notin H^1(M;\bZ)$. 
In this case one needs f{}irst to perturb $\sigma$ to an element in $H^1(M;\bQ)$ and then scale it to an element of $H^1(M;\bZ)$. Tischler's result, then, says that 
\emph{every} compact manifold $M$ endowed with a closed and nowhere vanishing 1-form $\sigma$ f{}ibres over the circle and is a mapping torus (see \cite{MT}).
 
A second situation in which a codimension 1 foliation is regular is the following:
 
 \begin{example}\label{compactness}
 Let $M^n$ be a compact manifold endowed with a closed and nowhere vanishing 1-form $\sigma$, so that $\cF=\ker\sigma$ is a foliation on $M$. Assume that $\cF$ has at least one compact leaf $L$. It is
shown in \cite{GMP} that 
 $M$ is the mapping torus of a dif{}feomorphism $\varphi\colon L\to L$. The f{}ibres of the mapping torus projection $M=L_\varphi\to S^1$ are dif{}feomorphic to $L$, and the local triviality of the
f{}ibration tells us that 
 we are in the regular case. 
\end{example}

We consider now the characteristic foliation, for which there is a regularity notion as well, and use it to prove a structure result for compact K-cosymplectic manifolds, 
which mimicks the contact case (see \cite{Boyer_Galicki}).

\begin{definition}
Let $(\eta,\xi,\phi,g)$ be an almost contact metric structure on $M$. The \textbf{characteristic foliation} is the 1-dimensional foliation $\cF_\xi$ whose leaves are given by the flow lines of $\xi$.
\end{definition}

\begin{remark}
 The characteristic foliation makes sense for cosymplectic manifolds $(M,\eta,\omega)$ as well, since the Reeb f{}ield $\xi$ is uniquely def{}ined by $\imath_\xi\omega=0$ and $\eta(\xi)=1$.
\end{remark}

\begin{definition}
Let $(\eta,\xi,\phi,g)$ be an almost contact metric structure and let $\cF_\xi$ be the characteristic foliation. The structure (or the foliation) is called \textbf{quasi-regular} 
if there exists a positive integer $k$ such that each point $p\in M$ has a 
neighborhood $U$ with the following property: any integral curve of $\xi$ intersects $U$ at most $k$ times. If $k=1$, the structure is called \textbf{regular}. We use the term \textbf{irregular} for
the non quasi-regular case.
\end{definition}

In Proposition \ref{prop:exquasireg} below we will show that if a compact manifold admits a K-cosymplectic structure, then it also admits a quasi-regular one.

The regularity of each foliation is independent from the regularity of the other, as the following two examples show.
\begin{example}
We consider a compact symplectic mapping torus $N_\varphi$, such that the symplectomorphism $\varphi\colon N\to N$ has an inf{}inite orbit. For example, take $N=T^2$ and $\varphi\colon T^2\to T^2$ to
be
the symplectomorphism covered by the linear map $\tilde{\varphi}\colon\bR^2\to\bR^2$ given by
\[
 \begin{pmatrix}
  2 & 1\\
  1 & 1
 \end{pmatrix}.
\]
The mapping torus $T^2_\varphi$ admits a cosymplectic strucure $(\eta,\xi,\phi,g)$ with $[\eta]\in H^1(T^2_\varphi;\bZ)$ (compare with Remark \ref{sympl_mapping_torus}).
Hence the vertical foliation is regular, according to Example \ref{Tischler}. The matrix $\tilde{\varphi}$
has inf{}inite order in the group of symplectomorphisms of $T^2$. Therefore, in $T^2_\varphi$, the characteristic foliation intersects each f{}ibre of the mapping torus f{}ibration an inf{}inite
number of times, hence it is not regular.
\end{example}

\begin{example}\label{ex:typeIIdeformation} We can also construct examples of compact cosymplectic manifolds for which the characteristic foliation is regular,
but the vertical foliation is not. To see this, consider the following analogue of a {\bf deformation of type II} in the Sasakian setting (see \cite{Boyer_Galicki}, p.\ 240).

Let $(M,\eta,\omega)$ be a cosymplectic manifold of dimension $2n+1$, with Reeb f{}ield $\xi$, and let $\beta\in \Omega^1(M,\cF_\xi)$ be an arbitrary closed basic $1$-form (i.e., $d\beta=0$ and
$\imath_\xi\beta=0$, see Section \ref{section:forms} below). 
Then $\eta'=\eta+\beta$ is again a closed $1$-form on $M$. Because $\imath_\xi\beta=0$ and
$\imath_\xi\omega=0$, we have $\beta\wedge \omega^n=0$, and hence $\eta'\wedge \omega^n=\eta\wedge\omega^n$ is a volume form on $M$. Thus, $(M,\eta',\omega)$ is again cosymplectic. The Reeb f{}ield of
this new cosymplectic manifold is equal to the original Reeb f{}ield $\xi$, but the vertical foliation has changed.

For an explicit example, take $T^3=T^2\times T^1$; let $\langle X_1,X_2,X_3\rangle$ be a basis of $\ft$ and let $\langle x_1,x_2,x_3\rangle $ be the dual basis. Set $\eta=x_3$
and $\omega=x_1\wedge x_2$. Let $g$ be the left-invariant Riemannian metric on $T^3$ which makes $\langle X_1,X_2,X_3\rangle$ orthonormal. Hence $(T^3,\eta,\omega)$ is a K-cosymplectic manifold. In
this case both the vertical and the characteristic foliation are regular. Since $\langle x_1,x_2,x_3\rangle$ is a basis of $H^1(T^3;\bR)$, any closed 1-form $\beta$ on $T^3$ can be written uniquely as
$\beta=\sum_{i=1}^3a_ix_i$ for some $a_i\in\bR$. $\beta$ is basic if and only if $a_3=0$. Choose $a_1,a_2$ to be algebraically independent over $\bQ$, set $\beta=a_1x_1+a_2x_2$ and $\eta'=\eta+\beta$.
Then $\eta'\notin\ H^1(T^3;\bZ)$; the perturbed structure $(T^3,\eta',\omega)$ is again K-cosymplectic, by Proposition \ref{Yamazaki}. The Reeb f{}ield has not changed, hence the characteristic
foliation is still regular; however, the vertical foliation is now irregular (it is dense in $T^3$). Compare also with Example \ref{irregular} below.
\end{example}

There is, however, a situation in which the regularity of the vertical foliation is related to that of the characteristic foliation. Let $K$ be a compact K\"ahler manifold,
let $\varphi\colon K\to K$ be a Hermitian isometry and let $K_\varphi$ be the corresponding mapping torus. Let $(\eta,\xi,\phi,g)$ be the natural
coK\"ahler structure on $K_\varphi$. Then $[\eta]\in H^1(K_\varphi;\bZ)$ (see again Remark \ref{sympl_mapping_torus}) and the vertical foliation is regular. It is proved in \cite[Theorem
6.6]{bazzoni_oprea} that such $\varphi$ has f{}inite order in the group of Hermitian isometries of $K$ (modulo the connected component of the identity). Therefore, in this specif{}ic case, the orbit
of $\varphi$ intersects each f{}ibre only a f{}inite number of times. Hence both the vertical and the characteristic foliation are regular.

We are almost ready for our structure result. In order to prove it, we need the following theorem, which is standard in almost contact metric geometry:
\begin{theorem}{\cite[Theorem 6.3.8]{Boyer_Galicki}}\label{Boyer-Galicki}
Let $(M,\xi,\eta,\phi)$ be an almost contact manifold such that the leaves of the characteristic foliation are all compact. Suppose also that $(M,\xi,\eta,\phi)$ admits a compatible Riemannian metric
$g$ such that $\xi$ is a Killing f{}ield which leaves $\phi$ invariant. Then the space of leaves $M/\cF_\xi$ has the structure of an almost Hermitian orbifold such that the canonical
projection $\pi\colon M\to M/\cF_\xi$ is an orbifold Riemannian 
submersion and a principal $S^1$ V-bundle over $M/\cF_\xi$ with connection 1-form $\eta$.
\end{theorem}

We refer to \cite[Chapter 4]{Boyer_Galicki}, and the references therein, for all the relevant def{}initions about orbifolds and V-bundles. However, we apply Theorem \ref{Boyer-Galicki} to the case of
a compact manifold $M$ endowed 
with a \emph{regular} K-cosymplectic structure, so that the space of leaves will be a manifold, and we can forget about orbifolds. We obtain:
\begin{theorem}\label{structure_2}
Let $M$ be a compact manifold endowed with a K-cosymplectic structure $(\xi,\eta,\phi,g)$. Assume that the characteristic foliation $\cF_\xi$ is regular.
Then the space of leaves $M/\cF_\xi$ has the structure of an almost K\"ahler manifold 
such that the canonical projection $\pi\colon M\to M/\cF_\xi$ is a Riemannian submersion and a principal, flat $S^1$-bundle with connection 1-form $\eta$.
\end{theorem}
\begin{proof}
First of all, regularity implies that the leaves of $\cF_\xi$ are all closed, hence compact. Thus they are homeomorphic to circles. Since the structure is K-cosymplectic, the Reeb f{}ield $\xi$ is
Killing with respect to $g$ and respects the tensor 
$\phi$ (these two conditions are in fact equivalent). Since $\xi$ also respects the K\"ahler form $\omega$ of the K-cosymplectic structure, the almost Hermitian manifold $M/\cF_\xi$ is indeed almost K\"ahler. $M$ is the total space of a
principal $S^1$-bundle $\pi\colon M\to M/\cF_\xi$, and $\eta$ is a connection 1-form. Since $\eta$ is closed, the bundle is flat.
\end{proof}


If $G$ is a compact Lie group, and $P\to B$ is a principal $G$-bundle over a smooth manifold $B$, it is well-known (see for instance \cite{St}) that $P$ has a connection with zero curvature if and only if $P$ is induced from the 
universal covering $\widetilde{B}\to B$ through a homomorphism $\pi_1(B)\to G$. In other words, flat $G$-bundles are determined by the \emph{monodromy} action of the fundamental group of the base on
$G$. In our case, when $\cF_\xi$ is regular, the flat $S^1$-bundle $M\to M/\cF_\xi$ is determined by a homomorphism $\pi_1(M/\cF_\xi)\to S^1$. Consider the following portion of the long exact sequence
of homotopy groups of the f{}ibration $M\to M/\cF_\xi$:
\begin{equation}\label{les}
\ldots\to \pi_1(S^1)\stackrel{\chi}{\to} \pi_1(M)\to \pi_1(M/\cF_\xi)\to \pi_0(S^1)\to \ldots
\end{equation}
\begin{lemma}
Let $M$ be a compact manifold endowed with a regular K-cosymplectic structure $(\xi,\eta,\phi,g)$. Then the map $\chi\colon \pi_1(S^1)\to\pi_1(M)$ in \eqref{les} is injective.
\end{lemma}
\begin{proof}
The compactness hypothesis on $M$ ensures that $\pi_1(M)$ always contains a subgroup isomorphic to $\bZ$. Indeed, by work of Li, any compact cosymplectic manifold $M$ is diffeomorphic to a mapping
torus (see \cite{Li}) and this displays the fundamental group of $M$ as a semi-direct product $\Gamma \rtimes \bZ$ (see also \cite{bazzoni_oprea}). This means that $H_1(M;\bZ)$ has rank at least 1.
Since the K-cosymplectic structure is regular, $M$ is endowed with a circle action given precisely by the flow 
of the Reeb f{}ield. In \cite{bazzoni_oprea} it is proved that the orbit map of this action is injective in homology. Therefore, it must also be injective in homotopy.
\end{proof}
When $\pi_1(M)=\bZ$, \eqref{les} becomes $0\to\bZ\stackrel{\chi}{\to}\bZ\to \pi_1(M/\cF_\xi)\to 0$. Hence $\chi$ is multiplication by some f{}ixed integer $k$ and
$\pi_1(M/\cF_\xi)\cong\bZ_k$. 
Therefore, there is only a f{}inite number of principal flat $S^1$-bundles over $M/\cF_\xi$. All such bundles become trivial when lifted to the universal cover of $M/\cF_\xi$.


\section{Dif{}ferential forms on K-cosymplectic manifolds}\label{section:forms}

Let $M$ be a smooth manifold and let $\cF$ be a smooth foliation; we consider the following subalgebras of smooth forms $\Omega^*(M)$:
\begin{itemize}
\item \emph{horizontal forms}: $\Omega^p_{\textrm{hor}}(M)=\{\alpha\in\Omega^p(M) \ | \ \imath_X\alpha=0 \ \forall \ X\in T\cF\}$;
\item \emph{basic forms}: $\Omega^p(M;\cF)=\{\alpha\in\Omega^p(M) \ | \ \imath_X\alpha=0=\imath_X d\alpha \ \forall \ X\in T\cF\}$;
\end{itemize}

The notation $\Omega^*_{\textrm{bas}}(M)$ is also common to indicate basic forms. The exterior dif{}ferential maps basic forms to basic forms, hence $(\Omega^*(M;\cF),d)$ is a dif{}ferential
subalgebra of $\Omega^*(M)$. The corresponding cohomology $H^*(M;\cF)$ is called the {\bf basic cohomology} of $\cF$.

\begin{lemma}\label{splitting}
Let $(M^{2n+1},\eta,\omega)$ be a cosymplectic manifold; let $\xi$ be the Reeb f{}ield and let $\cF_\xi$ be the characteristic foliation. Then 
\[
\Omega^p(M)=\Omega^p_{\mathrm{hor}}(M)\oplus\eta\wedge\Omega^{p-1}_{\mathrm{hor}}(M)
\]
as $C^\infty(M)-$modules, for $1\leq p\leq 2n+1$. Also, $C^\infty(M)=\Omega^0(M)=\Omega^0_{\mathrm{hor}}(M)$.
\end{lemma}
\begin{proof}
Since the tangent space to $\cF_\xi$ at $p\in M$ is spanned by $\xi_p$, $\Omega^p_{\textrm{hor}}(M)=\{\alpha\in\Omega^p(M) \ | \
\imath_\xi\alpha=0\}$.
Given any $\alpha\in\Omega^p(M)$, we can write
\[
\a=(\a-\eta\wedge\imath_\xi\a)+\eta\wedge\imath_\xi\a=\colon\a_1+\eta\wedge\a_2.
\]
Since $\eta(\xi)=1$, we see that $\imath_\xi\a_1=0$, hence $\a_1\in\Omega^p_{\mathrm{hor}}(M)$. Furthermore $\a_2\in\Omega^{p-1}_{\mathrm{hor}}(M)$, hence
\[
\Omega^p(M)=\Omega^p_{\mathrm{hor}}(M)+\eta\wedge\Omega^{p-1}_{\mathrm{hor}}(M)
\]
Suppose $\beta\in\Omega^p_{\mathrm{hor}}(M)\cap\eta\wedge\Omega^{p-1}_{\mathrm{hor}}(M)$. Hence $\eta\wedge\beta=0$; by contracting the latter with $\xi$, we get
$0=\beta-\eta\wedge\imath_\xi\beta=\beta$,
which gives $\beta=0$. 
\end{proof}

If $(M,\eta,\omega)$ is a cosymplectic manifold, the map $\fX(M)\to \Omega^1(M)$ def{}ined by
\begin{equation}\label{isomorphism:1}
 X\mapsto \imath_X\omega+\eta(X)\eta 
\end{equation}
is an isomorphism (see \cite[Proposition 1]{Albert}). By Lemma \ref{splitting},
\[
\Omega^1(M)=\Omega^1_{\mathrm{hor}}(M)\oplus\langle\eta\rangle,
\]
where $\langle\eta\rangle$ denotes the $C^\infty(M)$-module generated by $\eta$. We have $\imath_\xi(\imath_X\omega)=-\imath_X(\imath_\xi\omega)=0$, hence
$\imath_X\omega\in\Omega^1_{\mathrm{hor}}(M)$.
We can rephrase \eqref{isomorphism:1} in the following way:

\begin{proposition}\label{isomorphism}
Let $(M,\eta,\omega)$ be a cosymplectic manifold. Then the map 
\[
\begin{array}{rccc}
\Psi\colon & \fX(M) & \to & \Omega^1_{\mathrm{hor}}(M)\oplus C^\infty(M)\\
& X &\mapsto &(\imath_X\omega,\eta(X))
\end{array}
\]
is an isomorphism.
\end{proposition}

\subsection{Basic cohomology of K-cosymplectic structures}\label{cohomology K-cosymplectic}
We prove a splitting result for the de Rham cohomology of a compact K-cosymplectic manifold, generalizing the corresponding result for compact coK\"ahler manifolds (see \cite{BLO}).

\begin{theorem}\label{1-forms}
Let $(M,\eta,\omega)$ be a compact K-cosymplectic manifold. Then the cohomology $H^*(M;\bR)$ splits as $H^*(M;\cF_\xi)\otimes \Lambda\langle[\eta]\rangle$. In particular, for each $0\leq p\leq 2n+1$,
\[
H^p(M;\bR)=H^p(M;\cF_\xi)\oplus [\eta]\wedge H^{p-1}(M;\cF_\xi).
\]
\end{theorem}
\begin{proof}
 Let $(M,\eta,\omega)$ be a compact K-cosymplectic manifold. Set
\[
\Omega^p_\xi(M)=\{\alpha\in \Omega^p(M) \ | \ L_\xi\alpha=0\}.
\]
Since the Lie derivative commutes with the exterior derivative, $(\Omega^*_\xi(M),d)$ is a dif{}ferential subalgebra of $(\Omega^*(M),d)$.
In \cite[Corollary 4.3]{BLO} it is proven that if $M$ is coK\"ahler, the inclusion $\iota\colon(\Omega^*_\xi(M),d)\hookrightarrow (\Omega^*(M),d)$ is a quasi-isomorphism, i.e. it induces an
isomorphism in 
cohomology. The proof relies on the fact that $\eta$ is a parallel form in the coK\"ahler setting, and this remains true when the structure is K-cosymplectic (see Remark \ref{eta_parallel}). 
Alternatively, one can argue as follows. Since $M$ is compact, the closure of the Reeb flow generates a torus action on $M$ (see the discussion before Proposition \ref{Yamazaki}); the fact that
$\iota$ is a quasi-isomorphism is then a special case of a general result on the cohomology of invariant forms, see for instance 
\cite[\S 9, Theorem 1]{On}. Because $\Omega^*_\xi(M)\cap \Omega^*_{\mathrm{hor}}(M)=\Omega^*(M;\cF_\xi)$, Lemma \ref{splitting} implies that
\[
\Omega^p_\xi(M)=\Omega^p(M;\cF_\xi)\oplus\eta\wedge\Omega^{p-1}(M;\cF_\xi)
\]
for all $p$. Each summand of the right hand side is a dif{}ferential subalgebra; taking cohomology, the right-hand side gives $H^p(M;\cF_\xi)\oplus[\eta]\wedge H^{p-1}(M;\cF_\xi)$. By the above
discussion, 
the cohomology of the left-hand side is isomorphic to $H^p(M;\bR)$.
\end{proof}

As a consequence, we deduce some properties of the basic cohomology of the characteristic foliation on a K-cosymplectic manifold.
\begin{proposition}\label{basic_cohomology}
Let $(M,\eta,\omega)$ be a compact, connected K-cosymplectic manifold of dimension $2n+1$, let $\cF_\xi$ denote the characteristic foliation and let $H^*(M;\cF_\xi)$ be the basic cohomology. Then
\begin{enumerate}
\item the groups $H^p(M;\cF_\xi)$ are f{}inite dimensional;
\item $H^{2n}(M;\cF_\xi)\cong\bR$, $H^0(M;\cF_\xi)\cong\bR$ and $H^p(M;\cF_\xi)=0$ for $p>2n$;
\item the class $[\omega]^p\in H^{2p}(M;\cF_\xi)$ is non-trivial for $1\leq p \leq n$;
\item $H^1(M;\bR)\cong H^1(M;\cF_\xi)\oplus \bR[\eta]$;
\item there is a non-degenerate pairing
\[
\Psi\colon H^p(M;\cF_\xi)\otimes H^{2n-p}(M;\cF_\xi)\longrightarrow\bR.
\]
\end{enumerate}
\end{proposition}
\begin{proof}
A theorem of El Kacimi-Alaoui, Sergiescu and Hector \cite{EKASH} states that the basic cohomology of a Riemannian foliation on a compact manifold is always f{}inite-dimensional, which directly implies
1. However, in our simple situation, 
one does not need to invoke this result: since $M$ is a compact manifold, $b_p(M)=\dim H^p(M;\bR)<\infty$ for every $0\leq p\leq 2n+1$, and by Theorem \ref{1-forms}, $b_p(M)=\dim H^p(M;\cF_\xi)+\dim
H^{p-1}(M;\cF_\xi)$ which gives 
f{}initeness recursively. 

As $\cF_\xi$ is a foliation of codimension $2n$, the basic cohomology $H^*(M;\cF_\xi)$ vanishes in degrees larger than $2n$. Also, $b_0(M)=1$ implies that $H^0(M;\cF_\xi)\cong \bR$. The K\"ahler form $\omega$ is closed and non-degenerate, 
meaning that $\omega^p\neq 0$ for $1\leq p\leq n$. It is also basic, and cannot be exact by compactness of $M$, according to Stokes' theorem. 
Therefore $[\omega]^p$ is non-trivial in $H^{2p}(M;\cF_\xi)$ for $1\leq p \leq n$. Since $M$ is compact and $\eta\wedge \omega^n$ is a volume form, we have $b_{2n+1}(M)=1$, which implies that $H^{2n}(M;\cF_\xi)$ is generated by 
$[\omega]^n$. This completes the proof of 2.~and 3.

Number 4.~follows immediately from Theorem \ref{1-forms} and number 2. To construct the pairing $\Psi$ we proceed as follows. Given a cohomology class $[\alpha]\in H^p(M;\cF_\xi)$, by Poincaré duality on $M$ there exists $[\beta]\in 
H^{2n+1-p}(M;\bR)$ such that $([\alpha],[\beta])\neq 0$, where $(\cdot,\cdot)$ is the usual pairing on $M$. Now write $[\beta]=[\sigma]+[\eta\wedge\tau]$ according to the splitting of Theorem \ref{1-forms}, with $[\sigma]\in 
H^{2n+1-p}(M;\cF_\xi)$ and $[\tau]\in H^{2n-p}(M;\cF_\xi)$. Then $[\alpha\wedge\sigma]\in H^{2n+1}(M;\cF_\xi)$, which vanishes by number 2. Therefore we see that the Poincar\'e pairing on $M$ is given by
\[
\int_M\alpha\wedge\eta\wedge\tau.
\]
For such $\alpha$, therefore, the pairing $\Psi\colon H^p(M;\cF_\xi)\otimes H^{2n-p}(M;\cF_\xi)\longrightarrow\bR$
\[
\Psi([\alpha],[\tau])=\int_M\alpha\wedge\eta\wedge\tau
\]
is non-degenerate, and we have number 5.
\end{proof}
\begin{definition}
Given a compact K-cosymplectic manifold $(M,\eta,\omega)$ of dimension $2n+1$, we def{}ine the \textbf{basic Betti numbers} as
\[
b_p(M;\cF_\xi)=\dim H^p(M;\cF_\xi), \quad 0\leq p\leq 2n+1.
\]
\end{definition}
\begin{corollary}
Let $(M,\eta,\omega)$ be a compact K-cosymplectic manifold of dimension $2n+1$. Then
\begin{itemize}
\item $b_{2n-p}(M;\cF_\xi)=b_p(M;\cF_\xi)$;
\item For any $1\leq p\leq 2n$, the basic Betti numbers $b_p(M;\cF_\xi)$ are determined by
\[
b_p(M;\cF_\xi) = \sum_{i=0}^p (-1)^i b_{p-i}(M).
\]
In particular, they are topological invariants of $M$.
\end{itemize}
\end{corollary}

\subsection{The Lefschetz map on K-cosymplectic manifolds}\label{Lefschetz K-cosymplectic}

Let $(\eta,\xi,\phi,g)$ be a coK\"ahler structure on a manifold $M$ of dimension $2n+1$ and let $\omega$ be the K\"ahler form. Let $\cL\colon\Omega^p(M)\to\Omega^{2n+1-p}(M)$ be the map
\begin{equation}\label{Lefschetz}
\cL(\a)=\omega^{n-p}\wedge(\omega\wedge\imath_\xi\a+\eta\wedge\alpha). 
\end{equation}
$\cL$ is called the \textbf{Lefschetz map}. Unlike the symplectic Lefschetz map, $\cL$ does not commute with the exterior dif{}ferential, hence it does not descend to cohomology.
It is proven in \cite[Theorem 12]{CdLM} that, when $M$ is compact, $\cL\colon \cH^p(M)\to\cH^{2n+1-p}(M)$ is an isomorphism, where $\cH^*(M)$ denotes harmonic forms. If the structure $(\eta,\xi,\phi,g)$ is 
not coK\"ahler, then $\cL$ does not necessarily send harmonic to harmonic forms. Thus the Lefschetz map is not well def{}ined for arbitrary cosymplectic structures. It was observed in \cite{BLO} that
the restriction of the 
Lefschetz map to the dif{}ferential subalgebra $(\Omega^*_\xi(M),d)$ supercommutes with the exterior dif{}ferential. Furthermore, since the inclusion $\iota\colon(\Omega^*_\xi(M),d)\to(\Omega^*(M),d)$
is a quasi-isomorphism when the structure
is coK\"ahler and $M$ is compact (this is again \cite[Corollary 4.3]{BLO}), $\cL$ descends to the cohomology $H^*(M;\bR)$, and is shown to be an isomorphism, recovering the result of \cite{CdLM}. 

We observed before that $\iota\colon(\Omega^*_\xi(M),d)\to(\Omega^*(M),d)$ is a quasi-isomorphism also in the compact K-cosymplectic case. Hence we get the following

\begin{lemma}
 Let $(\eta,\xi,\phi,g)$ be a K-cosymplectic structure on a compact manifold $M$ of dimension $2n+1$. Then the Lefschetz map $\cL\colon\Omega^p_\xi(M)\to\Omega^{2n+1-p}_\xi(M)$ 
 is well def{}ined and descends to cohomology.
\end{lemma}

We are not claming here that $\cL$ is an isomorphism in the K-cosymplectic case. This is in fact false; to see this, notice that there exist symplectic manifolds which do not satisfy the usual
Lefschetz property (any symplectic non-toral nilmanifold does the job, see \cite{Benson_Gordon}). Choose one such $(K,\tau)$ and f{}ix an adapted metric $h$, so that $(K,\tau,h)$ is almost K\"ahler.
Then $M=K\times S^1$ admits a K-cosymplectic structure by Corollary \ref{product}. One can see that the Lefschetz map \eqref{Lefschetz} on $M$ is not an isomorphism.


\section{Cosymplectic and Hamiltonian vector f{}ields}\label{section:HamiltonianVF}

In this section we recall the notions of cosymplectic and Hamiltonian vector fields, which were introduced by Albert in \cite{Albert}; however, he used a slightly different terminology.
Furthermore, in his paper Albert deals with the cosymplectic and the contact case simultaneously. So far, only the contact set-up has caught attention. In order to make our
exposition self-contained, we allow ourselves to provide a proof of some of the results we quote.

\subsection{Cosymplectic vector f{}ields}
\begin{definition}
Let $(M,\eta,\omega)$ be a cosymplectic manifold. Let $\psi\colon M\to M$ be a dif{}feomorphism. $\psi$ is a \textbf{weak cosymplectomorphism} if $\psi^*\eta=\eta$ and there
exists a function $h_\psi\in C^\infty(M)$ such that 
$\psi^*\omega=\omega-dh_\psi\wedge\eta$. Such $\psi$ is a \textbf{cosymplectomorphism} if one can choose $h_\psi$ to vanish.
\end{definition}

\begin{remark}\label{Reeb_preserved}
A cosymplectomorphism $\psi$ satisf{}ies $\psi^*\eta=\eta$ and $\psi^*\omega=\omega$. Hence it respects the Reeb f{}ield and the characteristic foliation.
\end{remark}

\begin{definition}
Let $(M,\eta,\omega)$ be a cosymplectic manifold and let $X\in\fX(M)$ be a vector f{}ield. $X$ is \textbf{weakly cosymplectic} if $L_X\eta=0$ and there exists a function $h_X\in C^\infty(M)$ such that
$L_X\omega=-dh_X\wedge\eta$. 
$X$ is \textbf{cosymplectic} if one can choose $h_X$ to vanish.
\end{definition}
\begin{lemma}\label{Lemma:cosymplectic}
Let $X$ and $Y$ be two weakly cosymplectic  vector f{}ields such that $L_X\omega=-dh_X\wedge\eta$ and $L_Y\omega=-dh_Y\wedge\eta$ for $h_X,h_Y\in C^\infty(M)$.
Then $[X,Y]$ belongs to $\ker\eta$ and is weakly cosymplectic, with $h_{[X,Y]}=X(h_Y)-Y(h_X)$. In particular, if both $X$ and $Y$ are cosymplectic, then so is $[X,Y]$.
\end{lemma}
\begin{proof}
Suppose $X$ and $Y$ are weakly cosymplectic. Then $\eta(X)$ and $\eta(Y)$ are constant functions on $M$, and $d\eta=0$ implies $\imath_{[X,Y]}\eta=X\eta(Y)-Y\eta(X)=0$. We must prove that 
$L_{[X,Y]}\omega=d(\imath_{[X,Y]}\omega)=-d(X(h_X)-Y(h_X))\wedge\eta$:
\begin{align*}
d(\imath_{[X,Y]}\omega) &= dL_X\imath_Y\omega - d\imath_YL_X\omega= d\imath_X d\imath_Y\omega + d\imath_Y(dh_X\wedge \eta)= d\imath_X L_Y\omega + (d\imath_Ydh_X)\wedge \eta\\
&= -d\imath_X(dh_Y\wedge \eta) + (d(Y(h_X)))\wedge \eta= -d(X(h_Y)-Y(h_X))\wedge \eta.
\end{align*}
\end{proof}
\begin{corollary}
Let $(M,\eta,\omega)$ be a cosymplectic manifold. Then (weakly) cosymplectic vector f{}ields form a Lie subalgebra in $\fX(M)$.
\end{corollary}
The Reeb f{}ield $\xi$ is cosymplectic. We denote by $\fX^{\mathrm{cosymp}}(M)\subset\fX(M)$ the Lie subalgebra of cosymplectic vector f{}ields. For a cosymplectic vector f{}ield $X$, the function
$\eta(X)$
is constant; we also consider the following subset of $\fX^{\mathrm{cosymp}}(M)$:
\[
\fX^{\mathrm{cosymp}}_0(M)=\{X\in \fX^{\mathrm{cosymp}}(M) \ | \ \eta(X)=0\}.
\]


\subsection{Hamiltonian vector f{}ields}
\begin{definition}
Let $(M,\eta,\omega)$ be a cosymplectic manifold, let $\xi$ denote the Reeb f{}ield and let $X\in\fX(M)$ be a vector f{}ield. $X$ is \textbf{weakly Hamiltonian} if $\eta(X)=0$ and there exists $f\in
C^\infty(M)$ such that 
$\imath_X\omega=df-\xi(f)\eta$. $X$ is \textbf{Hamiltonian} if, in addition, $f$ can be chosen to be invariant along the flow of $\xi$, i.e. $\xi(f)=0$. In both cases, $f$ is the \textbf{Hamiltonian
function} of $X$.
\end{definition}
If $X\in\fX(M)$ is a weakly Hamiltonian vector f{}ield then $X$ is not, in general, cosymplectic, but only weakly cosymplectic, and the functions $h_X$ and $\xi(f)$ can be chosen to coincide. We
denote by $X_f$ the weakly Hamiltonian 
vector f{}ield such that $\imath_X\omega=df-\xi(f)\eta$ and by $\fX_w^{\mathrm{ham}}(M)$ the set of weakly Hamiltonian f{}ields. When $X\in\fX(M)$ is Hamiltonian, then it is cosymplectic, and the
1-form
$\imath_X\omega$ is exact. 
We denote by $\fX^{\mathrm{ham}}(M)\subset\fX_0^{\mathrm{cosymp}}(M)$ the subset of Hamiltonian vector f{}ields. 

\begin{proposition}\label{ideal:1}
Let $X,Y$ be cosymplectic vector f{}ields. Then $[X,Y]$ is Hamiltonian with Hamiltonian function $-\omega(X,Y)$.
\end{proposition}
\begin{proof}
 By Lemma \ref{Lemma:cosymplectic}, $\eta([X,Y])=0$. Since $\omega$ is closed, we have
 \begin{align}\label{eq:555}
  0&=d\omega(X,Y,Z)\nonumber \\
  &=X\omega(Y,Z)-Y\omega(X,Z)+Z\omega(X,Y)-\omega([X,Y],Z)+\omega([X,Z],Y)-\omega([Y,Z],X).
 \end{align}
 Since $Y$ is cosymplectic,
 \begin{equation}\label{eq:556}
  0=d(\imath_Y\omega)(X,Z)=-(\imath_Y\omega)([X,Z])+X(\imath_Y\omega)(Z)-Z(\imath_Y\omega)(X);
 \end{equation}
Plugging \eqref{eq:556} into \eqref{eq:555} we get
 \begin{align*}
  (\imath_{[X,Y]}\omega)(Z)&=-Y(\imath_X\omega)(Z)+(\imath_X\omega)([Y,Z])=-(L_Y(\imath_X\omega))(Z)=-(d(\imath_Y\imath_X\omega))(Z),
\end{align*}
where we used that $X$ is cosymplectic in the last equality. Hence $\imath_{[X,Y]}\omega=-d(\omega(X,Y))$. We need to check that $\omega(X,Y)$ satisf{}ies $\xi\omega(X,Y)=0$. By def{}inition
\[
\xi\omega(X,Y)=d(\omega(X,Y))(\xi)=\imath_\xi d(\omega(X,Y))=-\imath_\xi\imath_{[X,Y]}\omega=\imath_{[X,Y]}\imath_\xi\omega=0.
\]
\end{proof}

\begin{proposition}[\cite{Albert}, Proposition 3]\label{ideal}
Let $Z$ be a weakly cosymplectic vector f{}ield and let $X_f$ be a weakly Hamiltonian vector f{}ield. Then the following two formul\ae$ \ \!\!$  hold true:
\begin{itemize}
\item $[\xi,Z]=X_{h_Z}$;
\item $[Z,X_f]=X_{Z(f)}$.
\end{itemize}
\end{proposition}
\begin{proof}
Recall that the weakly Hamiltonian vector f{}ield $X_{h_Z}$ is def{}ined by the conditions $\eta(X_{h_Z})=0$ and $\imath_{X_{h_Z}}\omega=dh_Z-\xi(h_Z)\eta$. We prove that the same is true for
$[\xi,Z]$.
Now $\eta(Z)$ is constant, 
hence $\eta([\xi,Z])=0$. Furthermore,
\begin{align*}
\imath_{[\xi,Z]}\omega(Y)&=\omega([\xi,Z],Y)\stackrel{(\diamond)}{=}\omega(Z,[Y,\xi])+\xi\omega(Z,Y)=(\imath_Z\omega)([Y,\xi])+\xi(\imath_Z\omega)(Y)\\
&=-d(\imath_Z\omega)(Y,\xi)=(dh_Z\wedge\eta)(Y,\xi)=dh_Z(Y)-\xi(h_Z)\eta(Y),
\end{align*}
where $(\diamond)$ holds because $\omega$ is closed. For the second formula, it will be enough to prove that $\eta([Z,X_f])=0$ and $\imath_{[Z,X_f]}\omega=d(Z(f))-\xi(Z(f))\eta$. For the f{}irst
equality, we have
\[
\eta([Z,X_f])=Z\eta(X_f)-X_f\eta(Z)=0,
\]
since $\eta(X_f)=0$ and $\eta(Z)$ is a constant function.  Notice that 
\begin{align}\label{skew}
[\xi,Z](f)&= X_{h_Z}(f)=df(X_{h_Z})=\omega(X_f,X_{h_Z})=-dh_Z(X_f)=-X_f(h_Z).
\end{align} Next, recalling that $d\omega=0$, we have
\begin{align*}
(\imath_{[Z,X_f]}\omega)(Y)&=-\omega([X_f,Y],Z)-\omega([Y,Z],X_f)+Z\omega(X_f,Y)+X_f\omega(Y,Z)+Y\omega(Z,X_f)\\
&=-d(\imath_Z\omega)(X_f,Y)+df([Y,Z])-\xi(f)\eta([Y,Z])+Z(df(Y)-\xi(f)\eta(Y))\\
&=(dh_Z\wedge\eta)(X_f,Y)+Y(Z(f))-Z(\xi(f))\eta(Y)\\
&=X_f(h_Z)\eta(Y)+d(Z(f))(Y)-Z(\xi(f))\eta(Y)\\
&\stackrel{(\dag)}{=}[Z,\xi](f)\eta(Y)+d(Z(f))(Y)-Z(\xi(f))\eta(Y)\\
&=(d(Z(f))-\xi(Z(f))\eta)(Y),
\end{align*}
where $(\dag)$ holds in view of \eqref{skew}.
\end{proof}
\begin{corollary}\label{commutation}
Let $(M,\eta,\omega)$ be a cosymplectic manifold. Then 
\begin{itemize}
 \item $\xi$ is a central in $\fX^{\mathrm{cosymp}}(M)$;
 \item (weakly) Hamiltonian vector f{}ields form an ideal in (weakly) cosymplectic vector f{}ields.
\end{itemize}
\end{corollary}

Using this we obtain the following result:
\begin{lemma}\label{lem:cosymp0ideal}
$\fX^{\mathrm{cosymp}}_0(M)\subset \fX^{\mathrm{cosymp}}(M)$ is an ideal and there is a Lie algebra isomorphism
\begin{equation}\label{cosymplectic_vector_fields}
\fX^{\mathrm{cosymp}}(M)\cong\fX^{\mathrm{cosymp}}_0(M)\oplus\langle\xi\rangle
\end{equation}
\end{lemma}
\begin{proof}
The f{}irst statement is clear. For the second one, we consider the map $\fX^{\mathrm{cosymp}}(M)\to\fX^{\mathrm{cosymp}}_0(M)\oplus\langle\xi\rangle$ def{}ined by $X\mapsto
(X-\eta(X)\xi,\eta(X)\xi)$. Its inverse is simply 
$(X,\xi)\mapsto X+\xi$. Since $\xi$ is central in $\fX^{\mathrm{cosymp}}(M)$ and $\fX^{\mathrm{cosymp}}_0(M)$ is in ideal, the splitting \eqref{cosymplectic_vector_fields} is also of Lie algebras.
\end{proof}

We denote by $C^\infty_\xi(M)$ the subalgebra of $C^\infty(M)$ consisting of functions that are constant along the flow lines of $\xi$, namely
\[
C^\infty_\xi(M)=\{f\in C^\infty(M) \ | \ \xi(f)=0\}.
\]
\begin{remark}
Let $(M,\eta,\omega)$ be a compact K-cosymplectic manifold. We can characterize cosymplectic and Hamiltonian vector f{}ields in terms of the isomorphism $\Psi\colon \fX(M)\to
\Omega^1_{\textrm{hor}}(M)\oplus C^\infty(M)$ of Proposition 
\ref{isomorphism} as follows:
\begin{itemize}
\item $\fX^{\mathrm{cosymp}}(M)=\Psi^{-1}(\Omega^1_{\textrm{hor}}(M)\cap\ker d,\bR)$;
\item $\fX^{\mathrm{cosymp}}_0(M)=\Psi^{-1}(\Omega^1_{\textrm{hor}}(M)\cap\ker d,0)$;
\item $\fX^{\mathrm{ham}}(M)=\Psi^{-1}(d(C^\infty_\xi(M)),0)$.
\end{itemize}
\end{remark}
\begin{remark}
The subalgebra $C^\infty_\xi(M)$ is invariant under the action of the group of cosymplectomorphisms of $M$. This holds because $\psi$ respects the characteristic foliation by Remark
\ref{Reeb_preserved}.
\end{remark}

Let $(M,\eta,\omega)$ be a cosymplectic manifold. We want to f{}ind suf{}f{}icient conditions under which a cosymplectic vector f{}ield is Hamiltonian. If $M$ is
compact then $H^1(M;\bR)\neq 0$ (see Remark \ref{first_Betti_number}), hence we will not be able, in general, to solve the equation $d(\imath_X\omega)=0$, i.e. to f{}ind a function $f\in
C_\xi^\infty(M)$ such that $X_f=X$. However, we can handle some special cases.
\begin{proposition}\label{b_1=1}
Let $(M,\eta,\omega)$ be a compact cosymplectic manifold with $b_1(M)=1$. Then a vector f{}ield in $\fX^{\mathrm{cosymp}}_0(M)$ is Hamiltonian.
\end{proposition}
\begin{proof}
Since $b_1(M)=1$, $H^1(M;\bR)=\langle[\eta]\rangle$ by Proposition \ref{basic_cohomology}. As $X\in\fX^{\mathrm{cosymp}}_0(M)$, $\imath_X\omega$ is a closed
1-form, hence it def{}ines a cohomology class in $H^1(M;\bR)$. 
We have then $\imath_X\omega=a\eta-df$ for some $f\in C^\infty(M)$ and $a\in\bR$. 
We would like to show that $a=0$. Plugging the Reeb f{}ield in the last equation, we obtain $a=\xi(f)$. Fix a point $p\in M$ and 
let $\gamma=\gamma(t)$ be the integral curve of $\xi$ passing through $p$ at time $t=0$. Then
\[
\frac{d}{dt}\Big|_{t=0}(f\circ\gamma)(t)=df(\xi)=\xi(f)=a.
\]
This shows that $(f\circ\gamma)(t)=at+b$ for $b\in\bR$. Now, if $a\neq 0$, the image of $f\circ\gamma$ is an unbounded set in $\bR$, hence not compact; but this is absurd, since $M$ is compact. Therefore $a=0$.
\end{proof}

\begin{proposition}\label{criterion2}
Let $(M,\eta,\omega)$ be a compact K-cosymplectic manifold. Then there is an exact sequence of Lie algebras
\[
0\longrightarrow\fX^{\mathrm{ham}}(M)\stackrel{i}{\longrightarrow} \fX^{\mathrm{cosymp}}_0(M)\stackrel{\pi}{\longrightarrow} H^1(M;\cF_\xi)\longrightarrow 0,
\]
where $i$ is inclusion and $\pi(X)=[\imath_X\omega]$. Here $H^1(M;\cF_\xi)$ is intended as an abelian Lie algebra.
\end{proposition}
\begin{proof}
It is clear that $i$ is injective. We show that $\pi$ is well def{}ined. If $X\in \fX^{\mathrm{cosymp}}_0(M)$, then $\eta(X)=0$ and $\imath_X\omega$ is closed. Therefore $\imath_X\omega\in
H^1(M;\bR)$.
In view of Proposition 
\ref{basic_cohomology}, $H^1(M;\bR)=H^1(M;\cF_\xi)\oplus\bR[\eta]$. To show that $\pi(X)\in H^1(M;\cF_\xi)$ we must show that $\imath_\xi(\imath_X\omega)=0$. We have
$\imath_\xi(\imath_X\omega)=-\imath_X(\imath_\xi\omega)=0$. Take $X\in \fX^{\mathrm{ham}}(M)$; then $\imath_X\omega=df$ for some $f\in C^\infty_\xi(M)$, hence $\pi(X)=[df]=0$.
Take $X\in\fX^{\mathrm{cosymp}}_0(M)$ with $\pi(X)=0$. Then $\imath_X\omega=df$ with $\imath_\xi(\imath_X\omega)=0$; hence $\imath_\xi df=\xi(f)=0$ and $f\in C^\infty_\xi(M)$; thus
$X\in\fX^{\mathrm{ham}}(M)$ and $\ker(\pi)=\mathrm{im}(i)$. We prove surjectivity of $\pi$. Take $[\alpha]\in H^1(M;\cF_\xi)$; then $d\alpha=0$ and $\alpha(\xi)=0$. Since $\alpha\in\Omega^1(M;\bR)$,
there exists $X\in\fX(M)$ such that 
$\alpha=\imath_X\omega+\eta(X)\eta$; now $\alpha(\xi)=0$; therefore, contracting with $\xi$, we see that $\eta(X)=0$ and $\alpha=\imath_X\omega$. Since $\alpha$ is closed, $d(\imath_X\omega)=0$, hence
$X\in\fX^{\mathrm{cosymp}}_0(M)$ and 
$\pi$ is surjective. We are left with showing that $\pi$ is a Lie algebra homomorphism. By def{}inition we have $[\pi(X),\pi(Y)]=0$. On the other hand, by Proposition \ref{ideal:1},
\[
\pi([X,Y])=[\imath_{[X,Y]}\omega]=[-d(\imath_Y\imath_X\omega)]=0.
\]
\end{proof}

\begin{remark}
Proposition \ref{criterion2} follows from Propostion \ref{b_1=1} when the K-cosymplectic manifold $M$ has $b_1(M)=1$.
\end{remark}


\subsection{The Poisson structure}
Here we use the Lie algebra structure of (weakly) Hamiltonian vector f{}ields to induce a skew-symmetric bilinear form on $C^\infty(M)$.
\begin{definition}
Let $(M,\eta,\omega)$ be a cosymplectic manifold. The \textbf{Poisson bracket} is the bilinear map $\{\cdot,\cdot\}\colon C^\infty(M)\times C^\infty(M)\to C^\infty(M)$ def{}ined by
\[
\{f,g\}=X_f(g).
\]
\end{definition}
\begin{proposition}
The Poisson bracket is skew-symmetric and satisf{}ies the following properties:
\begin{enumerate}
\item $\{f+g,h\}=\{f,h\}+\{g,h\}$;
\item $\{f,g\}=-\omega(X_f,X_g)$; in particular, $\{\cdot,\cdot\}$ is skew-symmetric;
\item $[X_f,X_g]=X_{\{f,g\}}$;
\item $\{f,gh\}=\{f,g\}h+g\{f,h\}$;
\item $\{\{f,g\},h\}+\{\{g,h\},f\}+\{\{h,f\},g\}=0$.
\end{enumerate}
As a consequence, $(C^\infty(M),\{\cdot,\cdot\})$ is a Lie algebra and the map
\[
\begin{array}{rccc}
\Theta\colon & (C^\infty(M),\{\cdot,\cdot\}) & \to & (\fX_w^{\mathrm{ham}}(M),[\cdot,\cdot])\\
& f &\mapsto &X_f
\end{array}
\]
is a morphism of Lie algebras.
\end{proposition}
\begin{remark}
A proof of this Proposition can be found in \cite{Albert}; notice that we are using a dif{}ferent sign convention.
\end{remark}
\begin{corollary}\label{morphism_theta}
$C^\infty_\xi(M)\subset C^\infty(M)$ is a subalgebra. The map $\Theta$ induces a Lie algebra morphism
\[
\begin{array}{rccc}
\bar{\Theta}\colon & (C^\infty_\xi(M),\{\cdot,\cdot\}) & \to & (\fX^{\mathrm{ham}}(M),[\cdot,\cdot])\\
& f &\mapsto &X_f
\end{array}
\]
\end{corollary}
\begin{proof}
It is enough to prove that if $f,g\in C^\infty_\xi(M)$, so does $\{f,g\}$. This is checked as follows:
\begin{align*}
\xi(\{f,g\})&=\xi(X_f(g))=[\xi,X_f](g)=0,
\end{align*}
where we used Corollary \ref{commutation} in the last equality.
\end{proof}

We obtain a short exact sequence of Lie algebras
\[
0\longrightarrow\bR\longrightarrow C^\infty_\xi(M)\longrightarrow \fX^{\mathrm{ham}}(M)\longrightarrow 0,
\]
where $\bR$ represents constant functions.


\section{Deformations of cosymplectic structures}\label{section:deformations}

In this section we describe a type of deformation of (K-)cosymplectic manifolds, respectively structures, which are a generalization of the so-called {\bf deformations of type I}, see \cite[Section
8.2.3]{Boyer_Galicki}.
Deformations of this kind were introduced by Takahashi (see \cite{Takahashi}), in the particular context of Sasakian structures. Also, for the special case in which the vector f{}ield $\theta$ (see
below) is a scalar multiple of the Reeb f{}ield, they had been studied previously by Tanno (see \cite{Tanno}). 
Such deformations modify the characteristic foliation $\cF_\xi$ but preserve the distribution $\cD=\ker \eta$.

\subsection{Deformations of almost contact metric structures}

We will f{}irst describe the type of deformations for an arbitrary almost contact metric structure. 
Below we will apply it to (K-)cosymplectic and coK\"ahler structures; afterwards we also mention that they can be formulated in the setting of a (K-)cosymplectic manifold, without any f{}ixed
metric.

\begin{proposition}\label{deformation_1}
Let $(\eta,\xi,\phi,g)$ be an almost contact metric structure on a manifold $M$. Assume $\theta$ is a vector f{}ield on $M$ satisfying $1+\eta(\theta)> 0$. Set
\begin{eqnarray*}
\xi'& = & \xi+\theta\\
\eta' & = & \frac{\eta}{\eta(\xi')}\\
\phi' & = & \phi\circ(\Id-\eta'\otimes \xi')\\
g' & = & \frac{1}{\eta(\xi')}g\circ(\Id-\eta'\otimes \xi',\Id-\eta'\otimes \xi')+\eta'\otimes\eta'. 
\end{eqnarray*}
Then $(\eta',\xi',\phi',g')$ is again an almost contact metric structure on $M$. Its K\"ahler form is given by
\begin{equation}\label{deformed_Kahler_form}
\omega'=\frac{1}{\eta(\xi')}(\omega+(\imath_\theta\omega)\wedge\eta'). 
\end{equation}
\end{proposition}
\begin{proof}
First of all it is clear that $\eta'(\xi')=1$. Next we need to verify that $(\phi')^2=-\Id+\eta'\otimes\xi'$. Notice that $\eta'\circ\phi=0$, because $\ker \eta=\ker \eta'$. Thus we have
\begin{align*}
(\phi')^2(X)&=\phi'(\phi(X)-\eta'(X)\phi(\xi'))=\phi^2(X)-\eta'(X)\phi^2(\xi')\\
&=-X+\eta(X)\xi+\eta'(X)\xi'-\eta'(X)\eta(\xi')\xi\\
&=(-\Id+\eta'\otimes\xi')(X).
\end{align*}
This alone already implies that $\phi'(\xi')=0$. Next, $g'$ is a Riemannian metric on $M$. Indeed, 
\[
g'(X,X)=\frac{1}{\eta(\xi')}g(X-\eta'(X)\xi',X-\eta'(X)\xi')+\eta'(X)^2
\]
which is zero if and only if $X-\eta'(X)\xi'=\eta'(X)=0$, since $g$ is Riemannian. But this holds only if $X=0$. Last, we need to check the compatibility between $g'$ and $\phi'$. We compute
\begin{align*}
g'(\phi'(X),\phi'(Y))&=g'(\phi(X-\eta'(X)\xi'),\phi(Y-\eta'(Y)\xi'))\\
&=\frac{1}{\eta(\xi')}g(\phi(X-\eta'(X)\xi'),\phi(Y-\eta'(Y)\xi'))\\
&=\frac{1}{\eta(\xi')}g(X-\eta'(X)\xi',Y-\eta'(Y)\xi')\\
&=g'(X,Y)-\eta'(X)\eta'(Y).
\end{align*}
Finally, we compute the K\"ahler form of the deformed almost contact metric structure. We have $\eta'\circ \phi'=0$; also, note that $\phi(\xi')=\phi(\theta)$ and that $\imath_\xi\omega=0$. We obtain
\begin{align*}
\omega'(X,Y)&=g'(X,\phi'(Y))=\frac{1}{\eta(\xi')}g(X-\eta'(X)\xi',\phi(Y)-\eta'(Y)\phi(\xi'))\\
&=\frac{1}{\eta(\xi')}(\omega(X,Y)-\eta'(Y)g(X,\phi(\theta))-\eta'(X)g(\xi',\phi(Y))+\eta'(X)\eta'(Y)g(\xi',\phi(\xi')))\\
&=\frac{1}{\eta(\xi')}(\omega(X,Y)+\eta'(Y)(\imath_\theta\omega)(X)-\eta'(X)(\imath_\theta\omega)(Y)+\eta'(X)\eta'(Y)\omega(\xi',\xi'))\\
&=\frac{1}{\eta(\xi')}(\omega+(\imath_\theta\omega)\wedge\eta')(X,Y).
\end{align*}
This f{}inishes the proof.
\end{proof}

\begin{proposition}\label{deformation_2}
Let $(\eta,\xi,\phi,g)$ be a cosymplectic (resp. K-cosymplectic resp. coK\"ahler) structure on a manifold $M$. Assume $\theta$ is a vector f{}ield satisfying 
\begin{enumerate}
\item[i)] $1+\eta(\theta)> 0$;
\item[ii)] $[\xi,\theta]=0$;
\item[iii)] $L_\theta g=0=L_\theta\omega$.
\end{enumerate} 
Then the deformed structure $(\eta',\xi',\phi',g')$ is again 
cosymplectic (resp. K-cosymplectic resp. coK\"ahler).
\end{proposition}
\begin{proof}
Let us f{}irst assume that $(\eta,\xi,\phi,g)$ is cosymplectic and show that the same holds for $(\eta',\xi',\phi',g')$. It is enough to prove that $d\eta'=0=d\omega'$, where $\omega'$ is the K\"ahler
form of $(\eta',\xi',\phi',g')$.
The function $\eta(\theta)$ is constant on a cosymplectic manifold; indeed, 
since $\eta$ is closed, $d(\imath_\theta\eta)=L_\theta\eta$ and
\[
(L_\theta\eta)(X)=\theta(\eta(X))-\eta([\theta,X])=\theta(g(\xi,X))-g(\xi,[\theta,X])\stackrel{(\dag)}{=}g([\theta,\xi],X)\stackrel{(*)}{=}0
\]
where $(\dag)$ holds because $\theta$ is Killing and $(*)$ holds because $\theta$ and $\xi$ commute. Since $\eta(\theta)$ is constant, so is $\eta(\xi')$ and hence $d\eta'=\frac{1}{\eta(\xi')}d\eta=0$. We show next that $\omega'$ 
is closed. By \eqref{deformed_Kahler_form},
\[
d\omega'=\frac{1}{\eta(\xi')}d(\imath_\theta\omega)\wedge\eta'=0
\]
since $d\omega=0$ and $L_\theta\omega=0$ by hypothesis.

Now assume that $(\eta,\xi,\phi,g)$ is K-cosymplectic. In order to show that $(\eta',\xi',\phi',g')$ is K-cosymplectic, we use Corollary \ref{characterization} and prove that $L_{\xi'}\phi'=0$. Notice that since 
$(\eta,\xi,\phi,g)$ is K-cosymplectic, we have
\[
L_\xi\phi'=L_\xi\phi-L_\xi(\eta'\otimes\phi\xi')=-(L_\xi\eta')\otimes\phi\xi'+\eta'\otimes L_\xi(\phi\xi')=\eta'\otimes \phi(L_\xi\xi')=0.
\]
Moreover, $L_\theta \phi=0$ because $L_\theta g=L_\theta\omega=0$. Finally,
\[
L_\theta\phi'=L_\theta\phi-L_\theta(\eta\otimes\phi\xi')=\eta\otimes L_\theta(\phi\xi')=\eta\otimes \phi(L_\theta\xi')=0.
\]
To f{}inish we need to consider the case in which $(\eta,\xi,\phi,g)$ is coK\"ahler and show that $N_{\phi'}=0$, where $N_{\phi'}$ is the Nijenhuis torsion of $\phi'$. Since $N_{\phi'}$ is a tensor,
it is enough to compute it on pairs 
$(X,Y)$ and $(X,\xi')$, with $X,Y$ sections of $\ker \eta$. Take such $X$ and $Y$; then, because $\ker \eta$ is an integrable distribution,
\begin{align*}
N_{\phi'}(X,Y)&=(\phi')^2[X,Y]-\phi'[\phi'X,Y]-\phi'[X,\phi'Y]+[\phi'X,\phi'Y]\\
&=-[X,Y]-\phi[\phi X,Y]-\phi[X,\phi Y]+[\phi X,\phi Y]=N_\phi(X,Y)=0.
\end{align*}
Further, because the Lie bracket of $\xi'$ with a section of $\ker \eta$ is again in $\ker \eta$, we have 
\begin{align*}
N_{\phi'}(X,\xi')&=(\phi')^2[X,\xi']-\phi'[\phi'X,\xi']=-[X,\xi']-\phi[\phi X,\xi']=0
\end{align*}
using $L_{\xi'}\phi=0$ in the last equality.
\end{proof}

Of course, we can interpret this Proposition also in terms of cosymplectic manifolds, without having f{}ixed a Riemannian metric. Indeed, if $(M,\eta,\omega)$ is a cosymplectic manifold with Reeb
f{}ield
$\xi$, and $\theta\in\fX(M)$ is a cosymplectic vector f{}ield which commutes with $\xi$ and is such that $1+\eta(\theta)>0$, then
\[
\eta'=\frac{\eta}{1+\eta(\theta)}\quad \textrm{and} \quad \omega'=\frac{\omega+(\imath_\theta\omega)\wedge\eta'}{1+\eta(\theta)}
\]
def{}ine a new cosymplectic structure on $M$ with Reeb f{}ield $\xi'=\xi+\theta$.

\subsection{Cosymplectic group actions}

Cosymplectic group actions were introduced by Albert in \cite{Albert}. We recall the definition and prove a proposition which is needed in Section \ref{canonical_torus}.

\begin{definition}
Let $(M,\eta,\omega)$ be a cosymplectic manifold and let $G$ be a Lie group acting smoothly on $M$; we will denote the dif{}feomorphism of $M$ given by $g\in G$ by $g\colon M\to M;\, x\mapsto
g\cdot x$. Such an action is \textbf{cosymplectic} if, for every $g\in G$, 
\[
g^*\eta=\eta \quad \mathrm{and} \quad g^*\omega=\omega,
\]
In this case, we say that $G$ acts by cosymplectomorphisms on $M$.
\end{definition}
Let $\fg$ be the Lie algebra of $G$; for every element $A\in\fg$, let $\bar{A}\in\fX(M)$ denote the fundamental vector f{}ield of the action, def{}ined by
\[
\bar{A}(x)=\frac{d}{dt}\Big|_{t=0}\exp(tA)\cdot x.
\]
Then, if $G$ acts on $M$ by cosymplectomorphisms, one has
\[
L_{\bar{A}}\eta=0 \quad \mathrm{and} \quad L_{\bar{A}}\omega=0 \quad \forall A\in\fg.
\]
In particular, every fundamental vector f{}ield is cosymplectic.

\begin{corollary}\label{cor:complementint}
If $G$ acts on $M$ by cosymplectomorphisms, then the 1-form $\eta$ def{}ines a linear form on $\fg$, and $\fg'=\ker \eta$ is an ideal in $\fg$ which is either equal to $\fg$ or of codimension one.
\end{corollary}
\begin{proof}
Because the $G-$action is cosymplectic, we have  $d(\imath_{\bar{A}}\eta)=0$, hence $\eta(\bar{A})$ is a constant function. Thus $\eta$ is well def{}ined on $\fg$, and so is its kernel $\fg'$, which
has codimension $0$ or $1$. By 
Lemma \ref{lem:cosymp0ideal}, this kernel is in fact an ideal in $\fg$.
\end{proof}

Let us assume now that the acting Lie group $G$ is compact; this has the important consequence that each component of the f{}ixed point set is a submanifold of $M$.
With this fact in mind, we can prove the following:
\begin{proposition}\label{prop:fixedpointcosympsubmfd} Suppose $G$ is a compact Lie group acting cosymplectically on $M$. If $\eta(\bar A)=0$ for all $A\in \fg$, then every component of the f{}ixed
point set $M^G$ is a cosymplectic 
submanifold of $M$.
\end{proposition}
\begin{proof} The proof is analogous to the contact case, see \cite[Lemma 9.15]{GNT}. 
Let $N$ be a component of $M^G$, and $p\in N$. Then the tangent space $T_pN$ is exactly the subspace of $T_pM$ consisting of elements f{}ixed by the isotropy representation of $G$. In particular, this
shows $\xi_p\in T_pN$. 

We want to show that $(N,\eta|_N,\omega|_N)$ is a cosymplectic manifold. In other words, $\omega$ has to be nondegenerate on  $\ker (\eta|_N) = TN\cap \ker \eta$. To see this, we decompose, for $p\in N$, the tangent space
\[
T_pM = T_pN \oplus \bigoplus_\mu V_\mu
\]
into the weight spaces of the isotropy representation. More precisely, we f{}ix a maximal torus $T\subset G$; then, each $V_\mu$ is the complex one-dimensional vector subspace such that for $A\in
{\mathfrak{t}}$ and $v\in V_\mu$ we have $\bar{A}\cdot v = \mu(A)Jv$. Thus, for $v\in T_pN$ and $w\in V_\mu$ we have
\[
0 = (L_{\bar A}\omega)(v,w) = -\omega(\bar{A}\cdot v,w) - \omega(v,\bar{A}\cdot w) = -\mu(A)\omega(v,Jw).
\]
Hence, $T_pN$ is $\omega$-orthogonal to every $V_\mu$. This implies that $\omega$ is nondegenerate on $TN\cap \ker \eta$.
\end{proof}

\subsection{Deformations induced by cosymplectic actions}

Consider a cosymplectic action of a Lie group $G$ on a compact cosymplectic manifold $(M,\eta,\omega)$. The fundamental vector f{}ields of the action are cosymplectic; it follows from Corollary
\ref{commutation} that they commute with $\xi$. Thus we can apply the general deformation described in the previous section to any fundamental vector f{}ield
$\bar A$, where $A\in \fg$, as long as $1+\eta(\bar A)>0$, in order to obtain a new 
cosymplectic structure on $M$ with Reeb f{}ield $\xi+\bar A$.

\begin{example} Consider the standard symplectic structure on $\bC P^n$, together with the induced K-cosymplectic manifold $(M=\bC P^n\times S^1,\eta,\omega)$. On $\bC P^n$ we have a natural
$T^n$-action def{}ined by
\[
(t_1,\ldots,t_n)\cdot [z_1:\ldots:z_{n+1}] = [t_1z_1:\ldots:t_nz_n:z_{n+1}]
\]
and we consider the induced product action of $T^{n+1}=T^n\times S^1$ on $M$. This K-cosymplectic structure is regular, with the Reeb f{}ield being a fundamental f{}ield of the $S^1$-factor. The
action is cosymplectic, so by 
choosing $\bar A$ close to the original Reeb f{}ield (i.e., such that $\eta(\bar A)>0$) and such that the one-parameter subgroup
def{}ined by $A$ is
dense in $T^{n+1}$, we can f{}ind a new K-cosymplectic structure on $\bC P^n\times S^1$ with Reeb f{}ield $\bar A$; hence it has the property that each Reeb orbit is dense in the 
corresponding $T^{n+1}$-orbit.

We observe that the $T^{n+1}$-action has precisely $n+1$ one-dimensional orbits, given by $\{[1:0:\ldots:0]\}\times S^1, \cdots, \{[0:\ldots:0:1]\}\times S^1$. These coincide with the closed Reeb
orbits of the deformed structure. Compare 
also Corollaries \ref{cor:n+1orbits} and \ref{cor:exactlyn+1orbits} below.
\end{example}

We thus see that we can use the type of deformation introduced above in order to obtain examples of irregular cosymplectic manifolds. But we can also use it in the converse direction: Assume that we
are given an irregular K-cosymplectic structure on a compact manifold $M$. 
By \cite{MS}, the isometry group of $(M,g)$ is a compact Lie group. The Reeb f{}ield $\xi$ generates a $1-$parameter subgroup. Set
\[
T=\overline{\{\exp(t\xi) \mid t\in\bR\}}.
\]
Then $T$ is a torus acting on $M$, and the irregularity assumption means that its dimension is strictly larger than $1$; notice that, since $\xi$ preserves the K-cosymplectic structure, and its flow is by construction dense on $T$, by continuity every 
element $t\in T$ acts on $M$ preserving the K-cosymplectic structure. Inf{}initesimally, we thus obtain a Lie algebra homomorphism $\ft\to \fX^{\mathrm{cosymp}}(M)$.

For any element $A\in \ft$ such that $\eta(\bar A)>0$ we f{}ind a K-cosymplectic structure with $\bar A$ as Reeb f{}ield.
Choosing $A$ such that it generates a circle in $T$ we have shown:
\begin{proposition}\label{prop:exquasireg}
On any compact K-cosymplectic manifold there exists a K-cosymplectic structure which is either regular or quasi-regular.\end{proposition}


\section{Hamiltonian group actions}\label{section:HamiltonianGroupActions}

In this section we recall the notion of Hamiltonian action with the corresponding momentum map, as def{}ined by Albert in \cite{Albert}. We provide some existence and uniqueness statements for the
momentum map in the spirit of Hamiltonian group actions on symplectic manifolds.

\begin{definition}
Let $(M,\eta,\omega)$ be a cosymplectic manifold and let $G$ be a Lie group acting on $M$ by cosymplectomorphisms. The $G-$action is \textbf{Hamiltonian} if there exists a smooth map $\mu\colon M\to\fg^*$ such that
\[
\bar{A}=X_{\mu^A} \quad \forall A\in \fg,
\]
where $\mu^A\colon M\to \bR$ is the function def{}ined by the rule $\mu^A(x)=\mu(x)(A)$ and $X_{\mu^A}$ is the Hamiltonian vector f{}ield of this function. 
Furthermore, we require $\mu$ to be equivariant with respect to the natural coadjoint action of $G$ on $\fg^*$:
\[
\mu(g\cdot x)=g\cdot \mu(x)=\mu(x)\circ {\mathrm{Ad}}_{g^{-1}}.
\]
The map $\mu$ is called \textbf{momentum map} of the Hamiltonian $G-$action.
\end{definition}
By def{}inition of a Hamiltonian vector f{}ield, we have $\eta(X_{\mu^A})=0$. One sees easily that the conditions of the def{}inition imply that each component of the momentum map satisf{}ies
\[
\imath_{\bar A}\omega=d\mu^A.
\]
\begin{remark}
Not every group action on a cosymplectic manifold is Hamiltonian. Indeed, if $M$ is compact (or if the flow of the Reeb f{}ield $\xi$ is complete), then $M$ is endowed with a natural $\bR-$action,
given by the flow of $\xi$. 
The fundamental f{}ield of this action is $\xi$ and since $\eta(\xi)=1$, such an action will never be Hamiltonian. This is dif{}ferent from what happens in the contact case.
\end{remark}

It is possible to give another, equivalent, def{}inition of Hamiltonian action, which uses the so-called comomentum map.
\begin{definition}
Let $(M,\eta,\omega)$ be a cosymplectic manifold and let $G$ be a Lie group acting on $M$ by cosymplectomorphisms. The $G-$action is \textbf{weakly Hamiltonian} if there exists a map $\nu\colon \fg\to C^\infty_\xi(M)$ which makes the 
following diagram commute:
\begin{equation}\label{comoment_map}
\xymatrix{
\fg\ar[r]^\nu\ar[d] & C^\infty_\xi(M)\ar[d]^{\bar{\Theta}}  & \\
\fX^{\mathrm{cosymp}}(M) & \fX^{\mathrm{ham}}(M)\ar[l]}
\end{equation}
The map $\bar{\Theta}$ comes from Corollary \ref{morphism_theta}. The action is \textbf{Hamiltonian} if $\nu$ is an anti-morphism of Lie algebras. Such map $\nu$ is called \textbf{comomentum map} of the Hamiltonian $G-$action.
\end{definition}
\begin{lemma}\label{constant}
Let $(M,\eta,\omega)$ be a connected cosymplectic manifold and let $G$ be a Lie group acting on $M$ by cosymplectomorphisms. Assume that the action is weakly Hamiltonian. Then, for $A,B\in\fg$, the function 
$\nu([A,B])+\{\nu(A),\nu(B)\}$ is constant.
\end{lemma}
\begin{proof}
Since the action is weakly Hamiltonian, given $A\in\fg$, $\nu(A)\in C^\infty_\xi(M)$ is a Hamiltonian function for $\bar{A}$. Now we have $\overline{[A,B]}=-[\bar{A},\bar{B}]$, therefore $-\{\nu(A),\nu(B)\}$ is a Hamiltonian function 
for $\overline{[A,B]}$. Since the same is true for $\nu([A,B])$, their dif{}ference $\nu([A,B])+\{\nu(A),\nu(B)\}$ is constant.

\end{proof}
\begin{proposition}\label{equivalence}
Let $(M,\eta,\omega)$ be a connected cosymplectic manifold and let $G$ be a connected Lie group. Let $\Phi\colon G\times M\to M$ be a cosymplectic action of $G$ on $M$. The two notions of Hamiltonian $G-$action are equivalent.
\end{proposition}
\begin{proof}
Assume we have a momentum map $\mu\colon M\to\fg^*$ which is equivariant and such that, for every $A\in\fg$, $\xi(\mu^A)=0$. Def{}ine $\nu\colon \fg\to C^\infty_\xi(M)$ by $\nu(A)(x)=\mu^A(x)$, where
$x\in M$. Then $\nu$ is well def{}ined. For 
$A,B\in\fg$ and $x\in M$, we compute
\begin{align*}
\nu([A,B])(x)&=\mu^{[A,B]}(x)=\mu(x)([A,B])=\frac{d}{dt}\Big|_{t=0}\mu(x)(\mathrm{Ad}_{\exp(tA)}B)\\
&=\frac{d}{dt}\Big|_{t=0}\mathrm{Ad}^*_{\exp(tA)}(\mu(x))(B)=\frac{d}{dt}\Big|_{t=0}(\mu(\Phi_{\exp(-tA)}(x)))(B)\\
&=-d\mu^B(X_{\mu^A})(x)=-X_{\mu^A}(\mu^B)(x)=-\{\mu^A,\mu^B\}(x)\\
&=-\{\nu(A),\nu(B)\}(x).
\end{align*}
For the converse, suppose that we have a comomentum map $\nu\colon \fg\to C^\infty_\xi(M)$ which is an anti-morphism of Lie algebras. Def{}ine the map $\mu\colon M\to\fg^*$ in the following way:
$\mu(x)(A)=\nu(A)(x)$, for any $A\in\fg$. Also, for a f{}ixed $A\in\fg$, set $\mu^A(x)=\mu(x)(A)$. We prove f{}irst that $\xi(\mu^A)=0$ for every $A\in\fg$. Indeed, 
\[
\xi(\mu^A)=d\mu^A(\xi)=d(\nu(A))(\xi)=\xi(\nu(A))=0,
\]
since $\nu(A)\in C^\infty_\xi(M)$ for any $A\in\fg$. By diagram \eqref{comoment_map}, the vector f{}ield  $X_{\nu(A)}$ is Hamiltonian. Now $\nu(A)=\mu^A$, hence $\bar{A}=X_{\nu(A)}=X_{\mu^A}$.
Finally, we prove the equivariance of 
$\mu$. Since $G$ is connected, it is enough to prove this inf{}initesimally. We must therefore show that the map $\mu\colon M\to\fg^*$ we just def{}ined satisf{}ies, for every $x\in M$,
\[
d_x\mu(\bar{A}_x)=\hat{A}_{\mu(x)},
\]
where $\hat{A}$ denotes the inf{}initesimal vector f{}ield of the coadjoint action of $G$ on $\fg^*$, def{}ined by $\hat{A}_\alpha(B)=\alpha([B,A])$ for $\alpha\in\fg^*$ and $A,B\in\fg$.
We compute
\begin{align*}
\hat{A}_{\mu(x)}(B)&=\mu(x)([B,A])=\nu([B,A])(x)=\{\nu(A),\nu(B)\}(x)=X_{\nu(A)}(\nu(B))(x)\\
&=\bar{A}_x(\nu(B))=d_x\nu(B)(\bar{A}_x)=d_x\mu^B(\bar{A}_x)\\
&=d_x\mu(\bar{A}_x)(B).
\end{align*}
\end{proof}

\subsection{Existence and uniqueness of the momentum map}

We are interested in suf{}f{}icient conditions under which a $G-$action on a cosymplectic manifold $M$ is Hamiltonian. We can prove the existence of the momentum map in some simple cases which are of
our interest.
\begin{theorem}\label{thm:b1=1thenhamiltonian}
Let $(M,\eta,\omega)$ be a compact cosymplectic manifold with $b_1(M)=1$ and let $T$ be a torus which acts on $M$. Suppose that the $T-$action is cosymplectic and that $\eta(\bar{A})=0$ for every $A\in\ft$. Then the $T-$action is 
Hamiltonian.
\end{theorem}
\begin{proof}
Set $\ft=\mathrm{Lie}(T)$, take $A\in\ft$ and let $\bar{A}\in\fX^{\mathrm{cosymp}}_0(M)$ be the fundamental vector f{}ield of the action. By Proposition \ref{b_1=1}, $\bar{A}$ is Hamiltonian, hence,
$\imath_{\bar{A}}\omega=df^A$ for some 
function $f^A\in C^\infty(M)$; in fact, $f^A\in C^\infty_\xi(M)$. Next, we collect all these functions in a map $\nu\colon \ft\to C^\infty_\xi(M)$ as follows.
Choose a basis $\{A_1,\ldots,A_n\}$ of $\ft$, def{}ine $\nu(A_i)=f^{A_i}$ for $1\leq i\leq n$ and extend $\nu$ to $\ft$ by linearity. 
We proceed to show that $\nu$ is an anti-morphism of Lie algebras. The function $\{\nu(A),\nu(B)\}+\nu([A,B])$ is constant on $M$ (see Lemma \ref{constant}). Since the Lie algebra $\ft$ is abelian, this implies that 
$\{\nu(A),\nu(B)\}$ is constant. But $M$ is a compact manifold, so $\nu(A)$ must have at least one critical point. Hence $\{\nu(A),\nu(B)\}$ vanishes at some point, so it is identically zero. Then $\nu([A,B])=0=-\{\nu(A),\nu(B)\}$, 
i.e., $\nu$ is an anti-morphism of Lie algebras and the $T-$action is Hamiltonian.
\end{proof}

The above theorem gives conditions on the cosymplectic manifold $M$ for a cosymplectic torus action to be Hamiltonian. Next, we give conditions on a compact Lie group $G$ that ensure that a cosymplectic $G-$action is Hamiltonian. 
The kind of conditions we give mirror what happens in the symplectic case (see \cite{Audin,G-S}). 
\begin{theorem}
Let $(M,\eta,\omega)$ be a compact cosymplectic manifold and let $G$ be a connected, compact semisimple Lie group acting on $M$ by cosymplectomorphisms. Assume that $\eta(\bar{A})=0$ for every $A\in\fg$. Then the action is Hamiltonian 
and the momentum map $\mu\colon M\to\fg^*$ is unique.
\end{theorem}
\begin{proof}
The proof is formally equal to the symplectic case, where every symplectic action of a compact semisimple Lie group is Hamiltonian. One shows that the obstruction to lifting the map $\fg\to\fX_0^{\mathrm{cosymp}}(M)$ to a map 
$\fg\to\fX^{\mathrm{ham}}(M)$ lies in $H^2(\fg;\bR)$, which vanishes if $G$ is semisimple. We obtain therefore a comomentum map $\nu\colon \fg\to C^\infty_\xi(M)$, and the action is Hamiltonian, since we assume $G$ connected. The lack 
of uniqueness is measured by $H^1(\fg;\bR)$, which again vanishes since $G$ is semisimple.  
\end{proof}


\section{Closed Reeb orbits and basic cohomology}\label{section:ClosedReebOrbits}

In this section we derive a relation between the topology of the union $C$ of closed Reeb orbits on a compact $K$-cosymplectic manifold $M$ and the (basic) cohomology of $M$, similar to the
$K$-contact case treated in \cite{GNT}. As in \cite{GNT}, the proof follows from considering the equivariant cohomology of the torus action defined by the closure of the Reeb field, using the
statement that $C$ arises as the critical locus of a generic component of the momentum map, which was proven in the $K$-contact case by Rukimbira \cite{Ruk}. The analogous statement for fixed points
of Hamiltonian group actions on symplectic manifolds is well-known, see e.g.\ \cite[Proposition III.2.2]{Audin}.

\subsection{The canonical torus action on an irregular K-cosymplectic manifold}\label{canonical_torus}

Consider a Hamiltonian action of a torus $T$ on a cosymplectic manifold $(M,\eta,\omega)$, with momentum map $\mu\colon M\to \fg^*$. 
\begin{proposition}\label{prop:critmuA}
For any $A\in \ft$ we have ${\mathrm{crit}}(\mu^A) = \{p\in M\mid \bar{A}_p=0\}$.
\end{proposition}
\begin{proof}
Because $\bar{A}=X_{\mu^A}$, we have
\[
\imath_{\bar{A}}\omega = \imath_{X_{\mu^A}}\omega=d\mu^A,
\]
and hence the critical points of $\mu^A$ are precisely those points at which $\bar{A}$ is a multiple of the Reeb f{}ield $\xi$. But by def{}inition of a momentum map, we have
$\eta(\bar{A})=\eta(X_{\mu^A})=0$, hence at such a point $\bar{A}$ 
necessarily vanishes.
\end{proof}

Assume now that $M$ is a compact K-cosymplectic manifold. Let $T$ denote the torus given by the closure of the flow of the Reeb f{}ield; more precisely, we choose a metric with respect to which the
Reeb f{}ield is Killing, and def{}ine $T$ as the closure of the 1-parameter subgroup of the corresponding isometry 
group def{}ined by $\xi$. We additionally assume that our K-cosymplectic manifold is irregular, which in terms of $T$ just means that $\dim T\geq 2$. By Corollary \ref{cor:complementint}, the element
$\xi\in \ft$ has a canonical 
complement in $\ft$ given by $\fs=\ker \eta$. Let $S$ denote the connected Lie subgroup of $T$ with Lie algebra $\fs$. 

Because $T$ is def{}ined as the closure of the 1-parameter subgroup def{}ined by $\xi$, and $\xi$ is cosymplectic, the torus $T$ acts on $M$ by cosymplectomorphisms. Hence, $S$ also acts by
cosymplectomorphisms, and additionally satisf{}ies 
$\eta(\bar{A})=0$ for all $A\in \fs$. We now assume additionally that the $S$-action on $M$ is Hamiltonian (which by Theorem \ref{thm:b1=1thenhamiltonian} is automatic if $b_1=1$) and denote a momentum map by $\mu\colon M\to \fs^*$. 
We will argue below that, under this assumption, $S$ is closed in $T$. Notice that in general the corresponding statement for cosymplectic non-Hamiltonian actions is false:

\begin{example}\label{irregular}
We consider $(T^3,\eta,\omega)$ as a K-cosymplectic manifold, as we did in Example \ref{ex:typeIIdeformation}; let $\xi$ denote the Reeb f{}ield. Pick a vector $\xi'\in\ft$ with the property that the
closure of its flow is dense in $T^3\subset\mathrm{Isom}(T^3)$; for instance, one whose coef{}f{}icients are pairwise algebraically independent over $\bQ$ does the job. Using $\xi'$ we can perfom a
deformation of type I, as explained in Section \ref{section:deformations}, and obtain another K-cosymplectic manifold $(T^3,\eta',\omega')$. We use next a
deformation of type II to produce a third K-cosymplectic manifold $(T^3,\eta'',\omega'')$ with Reeb f{}ield $\xi''=\xi'$ such that the foliation $\ker\eta''$ is dense in $T^3$; more precisely, if $S$
is the unique connected Lie group with Lie algebra $\ker\eta''$, then $S$ is dense in $T$; in particular, $S$ is not a closed subtorus. The action of $T^3$ on itself is cosymplectic but not
Hamiltonian.
\end{example}

We denote by $C\subset M$ the union of all closed orbits of $\xi$, i.e., all flow lines which are homeomorphic to a circle. Note that by def{}inition, $C$ is equal to the union of all one-dimensional
$T$-orbits, and then also equal to the f{}ixed point set of the $S$-action. We then obtain the following statement about the functions $\mu^A\colon M\to \bR$, where $A\in \fs$:
\begin{proposition} For generic $A\in \fs$, i.e., such that the 1-parameter subgroup def{}ined by $A$ is dense in $S$, the critical set of $\mu^A$ is precisely $C$.
\end{proposition}
\begin{proof} By Proposition \ref{prop:critmuA} the critical set of $\mu^A$ is the set of points where $\bar{A}$ vanishes. But $\bar{A}$ vanishes at a point $p$ if and only if $p$ is a f{}ixed point
of the 1-parameter subgroup def{}ined 
by $A$, and these f{}ixed points coincide by assumption on $A$ with the f{}ixed points of $S$.
\end{proof}
Note that by compactness of $M$, the functions $\mu^A$ always have critical points. Thus, the proposition directly implies the existence of closed Reeb orbits on compact K-cosymplectic manifolds. Note that below, in Corollary 
\ref{cor:n+1orbits}, we will prove a more precise existence statement for closed Reeb orbits.

The fact that $C\neq \emptyset$ also implies that $S$ is closed in $T$, i.e., a subtorus: if the codimension one subgroup $S$ was not closed in $T$, then its closure would be equal to $T$. This would
imply that the $S$-f{}ixed point set 
equals the $T$-f{}ixed point set, but this contradicts the fact that the Reeb f{}ield has no zeros.

\begin{proposition} \label{prop:mumorsebott} For generic $A\in \fs$ the function $\mu^A$ is a Morse-Bott function.
\end{proposition}
\begin{proof}
The critical set of $\mu^A$ was shown to be the f{}ixed point set of the $S$-action; hence each of its components is a submanifold of $M$. We have to show that the Hessian of $\mu^A$ is non-degenerate
in the normal directions. The 
computation is structurally similar to the K-contact case, see \cite{Ruk}.

Let $p\in C$, and $v,w\in T_pM$ nonzero vectors normal to $C$. Extend $v$ and $w$ to vector f{}ields $V$ and $W$ in a neighborhood of $p$. Then we compute using $\bar A_p=0$:
\begin{align*}
{\mathrm{Hess}}_{\mu^A}(p)(v,w) &= V(W(\mu^A))(p) = V(d\mu^A(W))(p)  \\
&= V(\omega(\bar A,W))(p)= L_V(\omega(\bar{A},W))(p) \\
&= (L_V\omega)(\bar A,W)(p) + \omega([V,\bar A],W)(p) + \omega(\bar A,[V,W])(p) \\
&= (\imath_V\omega)([\bar A,W])(p) + \omega(\nabla_V\bar{A},W)(p)\\
&=-\omega(v,\nabla_w \bar A) + \omega(\nabla_v\bar A,w)\\
&=2\omega(\nabla_v\bar A,w).
\end{align*}

We now claim that $\nabla_v\bar A$ is always nonzero and perpendicular to $C$. For that we f{}ix a Riemannian metric on $M$ for which $\xi$ is Killing. Then the $T$-action is isometric, and all its
fundamental vector f{}ields are also Killing vector f{}ields. The restriction of the Killing f{}ield $\bar A$ to the geodesic $\gamma$ through $p$ in direction $v$ is a Jacobi f{}ield with initial
conditions $\bar A_p=0$ 
and $\nabla_v \bar A$; hence, if $\nabla_v \bar A$ was zero, then $\bar A$ would vanish along $\gamma$. But this would mean that $\gamma$ consists entirely of $S$-f{}ixed points, contradicting the
fact that $v$ points in a direction 
perpendicular to $C$. Hence, $\nabla_v \bar A\neq 0$. Moreover, we observe that 
\[
T_pC = \{u\in T_pM\mid [\bar A,u]=0\};
\]
note that the expression $[\bar A,u]\in T_pM$ is well-def{}ined because $\bar A$ vanishes at $p$. We then compute
\[
g(\nabla_v\bar A,u) = g([v,\bar A],u) = g(v,[\bar A,u]) = 0 
\]
for all $u\in T_pC$, using the facts that $\bar A$ is Killing and vanishes at $p$. Hence, $\nabla_v\bar A$ is perpendicular to $C$.

Because $C$ is the f{}ixed point set of the $S$-action all of whose fundamental vector f{}ields are in $\fX^{\mathrm{cosymp}}_0(M)$ (and because we have argued above that $S$ is indeed compact)
Proposition \ref{prop:fixedpointcosympsubmfd} 
shows that every component of $C$ is a cosymplectic submanifold. Hence, $\phi\nabla_v\bar A$ is also a nonzero vector perpendicular to $C$.
Then we obtain
\begin{align*}
{\mathrm{Hess}}_{\mu^A}(p)(v,\phi\nabla_v\bar A) = 2\omega(\nabla_v\bar A,\phi\nabla_v\bar A)\neq 0.
\end{align*}
\end{proof}

\subsection{Equivariant cohomology of the canonical action}
Recall that the (Cartan model of) {\bf equivariant cohomology} of an action of a compact Lie group $G$ on a manifold $M$ is def{}ined as the cohomology $H^*_G(M)$ of the complex $(C_G(M),d_G)$, where
\[
C_G(M) = (S(\fg^*)\otimes \Omega(M))^G
\]
is the space of {\bf $G$-equivariant dif{}ferential forms}, i.e., $G$-equivariant polynomials $\omega\colon \fg\to \Omega(M)$, and the equivariant dif{}ferential $d_G$ is def{}ined by
\[
(d_G\omega)(X) = d(\omega(X)) - \imath_{\bar X} (\omega(X)).
\]
The grading on $H^*_G(M)$ is def{}ined by imposing that a linear form on $\fg$ has degree two; more precisely, the equivariant dif{}ferential forms of degree $k$ are
\[
C_G^k(M) = \bigoplus_{2p+q=k}(S^p(\fg^*)\otimes \Omega^q(M))^G,
\]
and with this grading $d_G$ increases the degree by one.  The ring homomorphism
\[
S(\fg^*)^G\longrightarrow C_G(M);\, f\longmapsto f\otimes 1
\]
induces on $H^*_G(M)$ the structure of $S(\fg^*)^G$-algebra.
\begin{definition} The $G$-action is called {\bf equivariantly formal} if $H^*_G(M)\cong H^*(M;\bR)\otimes S(\fg^*)^G$ as $S(\fg^*)^G$-modules.
\end{definition}

Using the notation and assumptions of the previous subsection, we now consider the equivariant cohomology $H^*_{S}(M)$ of the $S$-action on the compact K-cosymplectic manifold $M$. We have:
\begin{theorem} If the $S$-action on a compact K-cosymplectic manifold is Hamiltonian, then it is equivariantly formal.
\end{theorem}
\begin{proof} The existence of a Morse-Bott function whose critical set is the f{}ixed point set of the action implies that the action is equivariantly formal; this was proven by Kirwan for the
momentum map of an 
Hamiltonian action on a compact symplectic manifold \cite[Proposition 5.8]{Kirwan} but is true for arbitrary Morse-Bott functions, see e.g.\ \cite[Theorem G.9]{GGK}. In our situation Propositions \ref{prop:critmuA} 
and \ref{prop:mumorsebott} show the existence of a Morse-Bott function with the desired property.
\end{proof}
For the action of a torus $T$ on a manifold $M$ we have that equivariant formality is equivalent to the equality of total Betti numbers $\dim H^*(M;\bR) = \dim H^*(M^T)$, where $M^T$ is the
$T$-f{}ixed point set, as follows 
from Borel's localization theorem, see e.g.\ \cite[Theorems C.20 and C.24]{GGK}. Applied to our canonical $S$-action this gives the following corollary:

\begin{corollary}\label{cor:dimC} For a compact K-cosymplectic manifold such that the canonical $S$-action is Hamiltonian we have $2\dim H^*(M;\cF_\xi) = \dim H^*(M;\bR) = \dim H^*(C;\bR)$.
\end{corollary}
\begin{proof}
The f{}irst equality follows directly from Theorem \ref{1-forms}. The second follows from the equivariant formality of the $S$-action, because the $S$-f{}ixed points are exactly $C$.
\end{proof}
In particular, if there are only f{}initely many closed Reeb orbits, then their number is given by $\dim H^*(M;\cF_\xi)$.

\begin{corollary} \label{cor:n+1orbits} A compact $2n+1$-dimensional K-cosymplectic manifold $M$ such that the canonical $S$-action is Hamiltonian has at least $n+1$ closed Reeb orbits.
\end{corollary}
\begin{proof}
By item 3.~of Proposition \ref{basic_cohomology} the $p-$th power of $\omega$ def{}ines a nontrivial class in $H^{2p}(M;\cF_\xi)$ for $1\leq p\leq n$. Thus, $\dim H^*(M;\cF_\xi)\geq n+1$.
\end{proof}
\begin{corollary} \label{cor:exactlyn+1orbits} 
Assume that the compact $2n+1$-dimensional K-cosymplectic manifold $M$ has only f{}initely many closed Reeb orbits, and that the canonical $S$-action is Hamiltonian. Then the following 
conditions are equivalent:
\begin{enumerate}
\item $M$ has exactly $n+1$ closed Reeb orbits.
\item $H^*(M;\cF_\xi)$ is generated by $[\omega]\in H^2(M;\cF_\xi)$.
\item $M$ has the real cohomology ring of $\bC P^n\times S^1$.
\end{enumerate}
\end{corollary}
\begin{proof}
The equivalence of 1.~and 2.~is just the fact that the number of closed Reeb orbits is equal to the dimension of $H^*(M;\cF_\xi)$. Conditions 2.~and 3.~are equivalent because by Theorem \ref{1-forms} we have 
$H^*(M;\bR)=H^*(M;\cF_\xi)\otimes \Lambda\langle[\eta]\rangle$ for any K-cosymplectic manifold.
\end{proof}

\begin{remark} Instead of considering ordinary $S$-equivariant cohomology, we could have also used, in the same way as in \cite{GNT}, the equivariant basic cohomology of the associated transverse action on the foliated manifold 
$(M,\cF_\xi)$. In this way we would directly obtain information on the basic cohomology $H^*(M;\cF_\xi)$. Our situation is, however, simpler than the K-contact case considered in \cite{GNT}, because in our K-cosymplectic case the 
union $C$ of closed Reeb orbits appears as the f{}ixed point set of the subtorus $S$ of $T$. 
\end{remark}

It is an interesting question to ask which compact manifolds have the same real cohomology of $\bC P^n\times S^1$. By the K\"unneth formula, it is suf{}f{}icient to f{}ind a manifold with the same
real
cohomology 
as $\bC P^n$. We are therefore looking for a smooth manifold $M$ of dimension $2n$ whose real cohomology is
\[
H^*(M;\bR)\cong \bR[x]/x^{n+1},
\]
where $x$ has degree 2. Such a manifold is called a rational or real cohomology $\bC P^n$.

\subsection{An example}

In this section we will construct, for any $m\geq 2$, an example of a real cohomology $\bC P^{2m-1}\times S^1$ with a minimal number of closed Reeb orbits which is not dif{}feomorphic to $\bC
P^{2m-1}\times S^1$. These examples can be considered as the coK\"ahler analogues of the Sasaki structures with minimal number of closed Reeb orbits on the Stiefel manifold considered in
\cite[Section 8]{GNT}.

Let $\cQ^{2m-1}$ denote the odd complex quadric, the zero locus of a single quadratic equation in $\bC P^{2m}$. As a homogeneous space, the quadric $\cQ^{2m-1}$ can be described as
\[
\cQ^{2m-1}=\frac{\SO(2m+1)}{\SO(2m-1)\times \SO(2)}.
\]
A good reference for this is \cite[Page 278]{KN}. As this is a homogeneous space of two compact Lie groups of equal rank, its real cohomology vanishes in odd degrees \cite[Volume III, Sec.\ 11.7, Theorem 7]{Greub3}. 
Moreover, its Euler characteristic is given by the quotient of the orders of the corresponding Weyl groups, see \cite[Volume II, Sec.\ 4.21]{Greub2}, and hence equal to $2m$. As $\cQ^{2m-1}$ is a K\"ahler manifold, it follows 
that the real cohomology ring of $\cQ^{2m-1}$ is precisely 
\[
H^*(\cQ^{2m-1};\bR)\cong \bR[x]/x^{2m},
\]
where $x$ is of degree $2$. Hence $\cQ^{2m-1}$ is a real cohomology $\bC P^{2m-1}$. However, $\cQ^{2m-1}$ and $\bC P^{2m-1}$ are not homeomorphic. Indeed, the odd quadric sits in the principal circle bundle 
$S^1\to V_2(\bR^{2m+1})\to \cQ^{2m-1}$, 
where $V_2(\bR^{2m+1})$ is the Stiefel manifold of orthonormal $2-$frames in $\bR^{2m+1}$. The long exact sequence of homotopy groups of this principal bundle gives 
\[
\pi_k(\cQ^{2m-1})=\pi_k(V_2(\bR^{2m+1})) \quad \mathrm{for} \ k\geq 3.
\]
One can see that $\pi_{\ell-j}(V_j(\bR^\ell))=\bZ_2$ for $j\geq 2$ and $\ell-j$ an odd number (\cite[Proposition 11.2]{H}); hence $\pi_{2m-1}(V_2(\bR^{2m+1}))=\bZ_2$ and also $\pi_{2m-1}(\cQ^{2m-1})=\bZ_2$. However, the long exact
homotopy sequence of the 
Hopf f{}ibration $S^1\to S^{4m-1}\to\bC P^{2m-1}$ shows that $\pi_{2m-1}(\bC P^{2m-1})=0$.

As $\cQ^{2m-1}$ is a K\"ahler manifold, $M=\cQ^{2m-1}\times S^1$ is a coK\"ahler manifold with the same real cohomology as $\bC P^{2m-1}\times S^1$. Since the structure is coK\"ahler, the flow of the
Reeb f{}ield consists of isometries. $M$ admits an 
isometric action of a torus $T=T^m\times S^1$ of dimension $m+1$. 
Using our deformation theory, we can perturb the coK\"ahler structure to one for which the flow of the Reeb f{}ield is dense in this torus $T$. The unique connected subgroup of $T$ with Lie algebra
$\ker \eta$ is $S:=T^m$. 
Since $b_1(M)=1$, the $S$-action is automatically Hamiltonian by Theorem \ref{thm:b1=1thenhamiltonian}. The closed Reeb orbits of the deformed structure are now given by the orbits through $S$-f{}ixed
points in $\cQ^{2m-1}$. A general fact about homogeneous spaces $G/H$ of two compact Lie groups of equal rank is now that the f{}ixed point set of the action of a maximal torus in $H$ on $G/H$ by left
multiplication is f{}inite, and given explicitly by the quotient of Weyl groups $W(G)/W(H)$. In our case, there are hence precisely $2m$ closed Reeb orbits, as predicted by Corollary 
\ref{cor:exactlyn+1orbits}.


 \end{document}